%% file: genre16.tex
\documentclass[11pt,a4paper]{article}
\usepackage{amsmath,amssymb,amsthm,enumerate,color,bm,hhline,makecell,array}

 \usepackage[alphabetic]{amsrefs}

\BibSpec{misc}{%
   +{}{\PrintAuthors}         {author}
   +{, }{ \textit}         {title}
   +{, }{}    {how}
}

\usepackage{hyperref}

\usepackage{chngcntr}
\usepackage{array,mathrsfs}
\usepackage[normalem]{ulem}
\usepackage{bm,mathtools}

\newcommand*{\email}[1]{%
    \normalsize\href{mailto:#1}{#1}\par
    }
\usepackage{authblk}

\usepackage[shortlabels]{enumitem}
\usepackage{colortbl}

\newlength{\elargir}
\usepackage{marginnote}

\allowdisplaybreaks
\let\mathcal\mathscr

\def\phi{\varphi}

\def\cI{\mathcal{I}}

\def\cB{\mathcal{B}}
\def\cC{\mathcal{C}}

\def\cF{\mathcal{F}}

\def\cI{\mathcal{I}}

\def\cL{\mathcal{L}}

\def\cO{\mathcal{O}}
\def\cP{\mathcal{P}}

\def\cT{\mathcal{T}}

\def\llra{\hbox to 10mm{\tofill}}
\def\lllra{\hbox to 15mm{\tofill}}

\def\llla{\hbox to 10mm{\leftarrowfill}}
\def\lllla{\hbox to 15mm{\leftarrowfill}}

\DeclareMathOperator{\Ext}{Ext}

\DeclareMathOperator{\Hilb}{Hilb}
\DeclareMathOperator{\Hom}{Hom}

\DeclareMathOperator{\Proj}{Proj}

\DeclareMathOperator{\rank}{rank}

\def\llra{\hbox to 10mm{\tofill}}
\def\lllra{\hbox to 15mm{\tofill}}

\newtheorem{lemm}{Lemma}[subsection]
\newtheorem{theo}[lemm]{Theorem}
\newtheorem*{theo*}{Theorem}
\newtheorem{coro}[lemm]{Corollary}
\newtheorem{prop}[lemm]{Proposition}
\newtheorem{claim}[lemm]{Claim}
\newtheorem{rema}[lemm]{Remark}

\theoremstyle{definition}
\newtheorem{defi}[lemm]{Definition}
\newtheorem{nota}[lemm]{Notation}
\newtheorem{algo}[lemm]{Algorithm}
\newtheorem{conj}[lemm]{Conjecture}

\newtheorem{exam}[lemm]{Example}

\theoremstyle{remark}
\newtheorem*{remark*}{Remark}
\newtheorem*{note*}{Note}

\def\Lkkk[#1]{{\Lambda_{\KKK^{[#1]}}}}
\def\kkk[#1]{{\KKK^{[#1]}}}
\DeclareMathOperator{\KKK}{{K3}}

\def\sss[#1]{{S^{[#1]}}}

\def\setminus{\smallsetminus}

\def\OO{\cO}

\def\CC{\mathbb{C}}
\def\ZZ{\mathbb{Z}}
\def\dual{{\vee}}

\def\Ytriv#1{Y^{#1}_{t_2}}
\def\Ypesk{Y^1_{t_2}}
\def\EVv{{E_{t_2}}} % le fibre noyau sur la variete de peskine pour t_2.
\def\Ypeskdual{Y^1_{t_1}}
\def\Btriv{\Ytriv{2}}
\def\KuVv{{K\!u_{t_2}}}

\RequirePackage{tikz-cd}
\RequirePackage{amssymb}
\usetikzlibrary{calc}
\usetikzlibrary{decorations.pathmorphing}

\tikzset{curve/.style={settings={#1},to path={(\tikztostart)
    .. controls ($(\tikztostart)!\pv{pos}!(\tikztotarget)!\pv{height}!270:(\tikztotarget)$)
    and ($(\tikztostart)!1-\pv{pos}!(\tikztotarget)!\pv{height}!270:(\tikztotarget)$)
    .. (\tikztotarget)\tikztonodes}},
    settings/.code={\tikzset{quiver/.cd,#1}
        \def\pv##1{\pgfkeysvalueof{/tikz/quiver/##1}}},
    quiver/.cd,pos/.initial=0.35,height/.initial=0}

\tikzset{tail reversed/.code={\pgfsetarrowsstart{tikzcd to}}}
\tikzset{2tail/.code={\pgfsetarrowsstart{Implies[reversed]}}}
\tikzset{2tail reversed/.code={\pgfsetarrowsstart{Implies}}}
\tikzset{no body/.style={/tikz/dash pattern=on 0 off 1mm}}
\def\Tauto#1{\cT_{#1}}
\def\Qmodular{{Q_4}} % le fibre de rang 4 sur S^[2]  Ogrady modulaire.

\def\PP{\mathbb{P}}
\def\FF{\mathbb{F}}
\def\Vv{V^\dual} % dim 4
\def\PW{\PP_9} %
\def\Wdix{{W_{10}}} % H^0(\OO_P9(1)))
\def\IX{\cI_X}
\def\NXv{{N_X^\dual}} % conormal de X dans P9
\def\NX{{N_X}} % normal de X dans P9
\def\mF{F} % fibre de Mukai de rang 2
\def\mG{{G_4}} % le mukai rigide de rang 4.
\def\mGdual{{G_4^\dual}} % le mukai rigide de rang 4.
\def\PF{\PP(\mF^\dual)}
\def\PVv{{\PP(V_{10}^\dual)}}
\def\PV{{\PP(V_{10})}}
\def\VV{{V_{10}}}
\def\VVv{{V_{10}^\dual}}
\def\sHom{{\mathcal Hom}}

\def\iX{{i_X}} % l'involution sur P9 associee au revetement double.
\def\fX{{f_X}} % la projection rationnelle de P9 vers P(V10dual)
\def\JacX{{Jac_{\fX}}} %le determinant de la jacobienne de f_X
\def\mgamma{m_{\gamma}} % l'application symmetrique de W_10*(-1) -> W_10 au dessus de \PV
\def\discV{D_\gamma}  % le discriminant
\def\ramiVv{D^*_\gamma} % la ramification
\def\nbmarge#1{\ifdefined\brouillon{{\color{blue} (#1)}}\fi}

\def\showmacaulay{} %% comment this line to hide macaulay code

\title{Geometry of genus sixteen K3 surfaces}

 \author[1]{Frédéric Han}
 \affil[1]{\small Université Paris Cité and Sorbonne Université, CNRS, IMJ-PRG, F-75013
   Paris, France.\newline \email{frederic.han@imj-prg.fr}}

\AtBeginDocument{%
   \def\MR#1{}
}

\begin{document}

\maketitle
\begin{abstract}
 Polarized K3 surfaces of genus sixteen have a Mukai vector bundle of rank two. We
 study the geometry of the projectivization of this bundle. We
 prove that it has an embedding in $\PP_9$ with an ideal  generated by
 quadrics. We give an effective method to compute these quadrics from a general choice in Mukai's
 unirationalization  of the moduli space.  This linear system gives a double
 cover of $\PP_9$ ramified on  a degree $10$ hypersurface. It gives relative
 Weddle/Kummer surfaces  over a Peskine variety associated to
 an explicit trivector. This work is also motivated by hyperkähler geometry and
 Debarre-Voisin varieties. Oberdieck showed that the Hilbert
 square of a general K3-surface of genus $16$ is a Debarre-Voisin variety for
 some trivector. We start to investigate the relationship between these two trivectors.
\end{abstract}

%%%%%%%%%%%
\section{Introduction}
Let $(S,h)$ be a general polarized K3 surface of genus $16$. All along this article, unless explicitly
specified such as in section~\ref{sectionchar2}, the base field will be $\CC$. In~\cite{mukaigenus16}
Mukai describes such a surface as the zero locus of general section of a vector
bundle in the Ellingsrud-Piene-Str\o mme compactification of the space of
of smooth rational cubic curves in $\PP_3$.
%NB c'est dominé par la composante de Hilb P3(3m+1) de dim 12.

In particular Mukai  proved the following

\begin{theo*}[{\cite[Prop~1.3,~2.2]{mukaigenus16}}]
  A generic polarized $K3$ surface $(S,h)$ of genus $16$ carries a rank $2$ vector
  bundle $\mF$ such that
  \begin{itemize}
  \item[---] $c_1(\mF)=h$, $c_2(\mF)=9$, $h^0(\mF)=10$, and $h^1(\mF)=h^2(\mF)=0$.
  \item[---] $\mF$ is simple, rigid and globally generated.

  \end{itemize}
 \end{theo*}
In \cite{DHOV} we introduced the following threefold to describe the rank four
modular (cf. \cite[1.2.1]{ogmodularsheaves}) vector bundle on $\Hilb^2(S)$.
\begin{defi}\label{refdefHh}
   Denote by $\Wdix$ the $10$ dimensional vector space $H^0(\mF)$ and $\PW$ the
   projective space $\PP(\Wdix^\dual)$. Define
   the threefold $X$ as the image of $\PF= \Proj(Sym(\mF))$ in $\PW$ and by $H$
   the hyperplane class of $\PF$.
   \[
    |\OO_{\PF}(H))|\colon~ \PF \xrightarrow{\quad \quad } X \subset \PW
   \]
\end{defi}
So $X$ depends only of $(S,h)$ and turns out to be in a quite small projective
space with equations in low degree. This is interesting because in genus sixteen
it is not obvious to get explicit equations  from Mukai's construction. For
instance up to now, this case is not implemented in the K3-surface package of
the open source software system for research in algebraic geometry \cite[Macaulay2]{macaulay2}.

In section~2 we give an effective construction for the equations of $X$ and prove
the following
\begin{theo*}[Th~\ref{refequationsX}, Cor~\ref{refXsmooth}, Cor~\ref{refdoublecover}]
  The threefold $X$ is defined (as a scheme) in $\PW$ by ten linearly
  independent quadrics and is an embedding of $\PF$. Moreover this linear system
  of quadrics gives a rational map of degree $2$ between two nine dimentional projective spaces.
\end{theo*}
Considering $S$ as a family of lines in $X$ we got the equations of $S$ in
 $\PP_{16}$ from the equations of $X$ and obtained with the help of the software
 \cite[Macaulay2]{macaulay2} the following
\begin{theo*}[Th~\ref{refequationsS}]
A generic complex K3-surface $S$ of genus sixteen is defined as a subscheme of
$\PP_{16}$ by the restriction of the Plücker equations of the Grassmannian
$G(2,10)$. In other words in the Plücker space we have $S=\PP_{16}\cap G(2,10)$ as scheme.
\end{theo*}
Section~2 ends  with the construction of a double cover
from the blow-up of $\PP_9$ along $X$ to another projective space of dimension
nine (Corollary~\ref{refdoublecover}).

\bigskip
In section~3 we define a trivector $t_2$ on a $10$-dimensional vector
space. This definition is explicit from the linear syzygies of $X$. Moreover
$t_2$ is also characterized by the family of invariant lines of the double cover
defined by $|\cI_X(2)|$.
Unfortunately we didn't undersand its Debarre-Voisin variety in terms of $S$.

\bigskip
Section~4 gather results from \cite{DV}, \cite{DHOV}, \cite{ogmodularsheaves},
\cite{oberdieckGW} to obtain a trivector $t_1$ such that its Debarre-Voisin
variety is isomorphic to the Hilbert square of $S$. These two trivectors are on
dual spaces, and this time we don't have an explicit construction of the trivector. We
just give a method to compute it from the embedding of $X$ in
$\PW$. Computations in \cite[Macaulay2]{macaulay2} suggest that the Peskine variety of
$t_1$ is a component of the singular locus of the discriminant of this linear
system of quadrics. On the other hand, we proved in  Proposition~\ref{refkummer} that
the Peskine variety for $t_2$ is in the singular locus of the
ramification. This section ends with more  results and conjectures on the
discriminant and  the ramification.

%% removed, not expected anymore%%
%%
% This last section ends with conjectures and remarks. In particular is still
% plausible that  we may have missed some natural isomorphism between these two dual spaces that
% whould exange $t_1$ and $t_2$.
%%%%%%%%%%%%

In Section~5 we give an example defined over $\FF_8$ to show that the effective
method of Section~2 is still usefull in characteristic 2 and that  many
statements are still valid.

\subsection*{Acknowledgements}
I would like to thank Junyu Meng for his detailed comments on this article.
\section{Equations of $X$ via Mukai's unirationalization}

\subsection{Initial choices in Mukai construction}
Let $V:= H^0(\OO_{\PP_3}(1))$, denote by $S_i\Vv$ the $i^{th}$-symmetric tensor product
of $\Vv$  and consider the following decompositon of irreducible
representation of $SL(\Vv)$
\[
  S_2 \Vv \otimes \Vv =  S_{2,1}\Vv \oplus S_3\Vv.
\]
Chose
\begin{enumerate}[label=\roman*)]
\item $M$ a general two dimensional linear subspace of $S_2 \Vv$.
\item $N$ a general two dimensional linear subspace of $S_{2,1} \Vv$.
\end{enumerate}

Then Mukai's construction gives a family of cubic curves in $\PP_3$ such that:
\begin{itemize}
\item[---] There is a polarized $K3$ surface $(S,h)$ of genus $16$ with Mukai
  rank $2$ vector bundle $F$ such that:

  \begin{enumerate}[label=\roman*)]
  \item $S$ has a (non rigid) rank
  $3$ vector bundle $E$ with a natural map
  \[
  E \boxtimes \OO_{\PP_3}(-1) \longrightarrow \mF\boxtimes \OO_{\PP_3}.
  \]
  \item The incidence point/curve $\cC_S$ of this family is defined by the degeneracy locus in
  $S\times \PP_3$ of this map.
  \end{enumerate}
   Let $\cL$ be the cokernel of this map, then  $\cC_S$ is the support of~$\cL$
   and we have following exact sequence
  \begin{equation}\label{refunivcubi}
    0 \longrightarrow \OO_{S\times \PP_3}(0,-3)  \longrightarrow    E \boxtimes
    \OO_{\PP_3}(-1) \longrightarrow \mF\boxtimes \OO_{\PP_3} \longrightarrow
    \cL_{} \longrightarrow 0
  \end{equation}
\item[---] The vector bundle $E$ is globally generated
  (cf. \cite[proof of Prop~2.2]{mukaigenus16}) with ${c_1(E)=h}$, $c_2(E)=13$.
 \item[---] The spaces of global sections of $E$ and $F$ fit in the natural
   exact sequences
   \begin{equation}\label{seqsectionE}
0 \longrightarrow M  \longrightarrow S_{2}\Vv \longrightarrow H^0(E) \longrightarrow  0
   \end{equation}
   \begin{equation}\label{seqsectionF}
0 \longrightarrow (\Vv\otimes M) \oplus N \longrightarrow S_{2,1}\Vv \longrightarrow W_{10} \longrightarrow  0,
\end{equation}
where $W_{10}=H^0(\mF)$.
\end{itemize}
\begin{rema}\label{refCStilde}
Denote by $\widetilde{\cC_S}$ the vanishing locus in $\PF\times
\PP_3$ of the cosection of
\[(p_S^*E)(-H)\boxtimes \OO_{\PP_3}(-1) \longrightarrow \OO_{\PF \times \PP_3}\]
defined by sequence~(\ref{refunivcubi}). The projection from $\widetilde{\cC_S}$ to
$\PP_3$ has degree given by $c_3(E^\dual(H))$, so it is $4$.
\end{rema}
In the next section\footnote{cf. sequence~(\ref{seqbetaprime})} we see that
$\widetilde{\cC_S}$ is a blowup of $\PF$ and that $\OO_{\PP_3}(1)$ has degree $3$ on
  the fibers of $p_S$.
% ce nombre est le c_3(E*(H)) qui vaut 4.
% -- cf genre16-degreinvolution.m2
% use S
% E3=abstractSheaf(S,Rank=>3,ChernClass => 1+h+13*p)  -- le choix dans mukai

% classeC=(chern(2,OO_X^4 - tensor(E3,OO_X(-1))))
% -- i68 : classeC=(chern(2,OO_X^4 - tensor(E3,OO_X(-1))))

% -- o68 = 2h*H  - 37p
% --           1

% --                RS[a , H ]
% --                    1   1
% -- o68 : ----------------------------
% --       (- a  - H  + h, - a H  + 9p)
% --           1    1         1 1

% classeC*H_1 -- C est de degre 23 dans P3
% classeC*h -- p_S(C) ~ 2h
% ---
% -- le degré de la projection de la courbe universelle vers P3 est 4.
% chern(3,tensor((dual E3),OO_X(1)))
% -- i35 : chern(3,tensor((dual E3),OO_X(1)))

% -- o35 = 4p*H
% --           1

% --                RS[a , H ]
% --                    1   1
% -- o35 : ----------------------------
% --       (- a  - H  + h, - a H  + 9p)
% --           1    1         1 1

\subsection{Explicit construction }
\nbmarge{ici c'est valable sur n'importe quel corps?}
\begin{lemm}\label{refdefalpha}
  Let $\cL$ be the sheaf defined in  sequence~(\ref{refunivcubi}) and $p_2$ 
  the projection from $S\times \PP_3$ to $\PP_3$. We have the following   exact
  sequences
  \begin{equation}
    \label{seqresolT}
0 \longrightarrow \OO_{\PP_3}(-3) \to \frac{S_2\Vv}{M} \otimes \OO_{\PP_3}(-1)
\overset{\alpha}{\longrightarrow} W_{10} \otimes \OO_{\PP_3} \to T \to 0
  \end{equation}
\[
0 \to T \to (p_{2})_* (\cL )\to \OO_{\PP_3}(-3) \to 0
\]
where  $T$ is a rank $3$ vector bundle on $\PP_3$. Moreover we have
$R^i(p_{2})_* (\cL ) =0$ for $i>0$.

The map $\alpha$ is obtained explicitely via the following compositions
% https://q.uiver.app/#q=WzAsOCxbMCwxLCJTXzJcXFZ2XFxvdGltZXMgXFxPT197XFxQUF8zfSgtMSkiXSxbMCwwLCJNXFxvdGltZXMgXFxPT197XFxQUF8zfSgtMSkiXSxbMiwwLCIoTVxcb3RpbWVzIFxcVnZcXG90aW1lcyBcXE9PX3tcXFBQXzN9KVxcb3BsdXMgKE5cXG90aW1lc1xcT09fe1xcUFBfM30pIl0sWzIsMSwiU197MiwxfVxcVnZcXG90aW1lcyBcXE9PX3tcXFBQXzN9Il0sWzEsMSwiU197Mn1cXFZ2XFxvdGltZXMgXFxWdlxcb3RpbWVzIFxcT09fe1xcUFBfM30iXSxbMSwwLCIoTVxcb3RpbWVzIFxcVnZcXG90aW1lcyBcXE9PX3tcXFBQXzN9KSJdLFswLDIsIntcXGZyYWN7U18yXFxWdn17TX0gXFxvdGltZXMgXFxPT197XFxQUF8zfSgtMSl9Il0sWzIsMiwiV197MTB9XFxvdGltZXMgXFxPT197XFxQUF8zfSJdLFsxLDAsIiIsMCx7InN0eWxlIjp7InRhaWwiOnsibmFtZSI6Imhvb2siLCJzaWRlIjoidG9wIn19fV0sWzIsMywiIiwyLHsic3R5bGUiOnsidGFpbCI6eyJuYW1lIjoiaG9vayIsInNpZGUiOiJ0b3AifX19XSxbMCw0XSxbNCwzXSxbMSw1XSxbNSwyXSxbNSw0LCIiLDEseyJzdHlsZSI6eyJ0YWlsIjp7Im5hbWUiOiJob29rIiwic2lkZSI6InRvcCJ9fX1dLFs2LDcsIlxcYWxwaGEiXSxbMCw2LCIiLDAseyJzdHlsZSI6eyJoZWFkIjp7Im5hbWUiOiJlcGkifX19XSxbMyw3LCIiLDAseyJzdHlsZSI6eyJoZWFkIjp7Im5hbWUiOiJlcGkifX19XV0=
\[
  \begin{tikzcd}[column sep=small]
	{M\otimes \OO_{\PP_3}(-1)} & {(M\otimes \Vv\otimes \OO_{\PP_3})} & {(M\otimes \Vv\otimes \OO_{\PP_3})\oplus (N\otimes\OO_{\PP_3})} \\
	{S_2\Vv\otimes \OO_{\PP_3}(-1)} & {S_{2}\Vv\otimes \Vv\otimes \OO_{\PP_3}} & {S_{2,1}\Vv\otimes \OO_{\PP_3}} \\
	{{\frac{S_2\Vv}{M} \otimes \OO_{\PP_3}(-1)}} && {W_{10}\otimes \OO_{\PP_3}}
	\arrow[hook, from=1-1, to=2-1]
	\arrow[hook, from=1-3, to=2-3]
	\arrow[from=2-1, to=2-2]
	\arrow[from=2-2, to=2-3]
	\arrow[from=1-1, to=1-2]
	\arrow[from=1-2, to=1-3]
	\arrow[hook, from=1-2, to=2-2]
	\arrow["\alpha", from=3-1, to=3-3]
	\arrow[two heads, from=2-1, to=3-1]
	\arrow[two heads, from=2-3, to=3-3]
      \end{tikzcd}
\]
\end{lemm}
\begin{proof}
  Let $C_{\bullet}$ be the resolution of $\cL$ defined by
  sequence~(\ref{refunivcubi}). For
  all $i>0$ we have the vanishing $0=h^i(E)=h^i(\mF)$,  so the hypercohomology spectral sequence
  ${E_{1}^{p,q}=H^q(C_p)} \Rightarrow R^{p+q}(p_2)_*(\cL)$  starts with

%suite spectrale
% https://q.uiver.app/#q=WzAsMTMsWzMsMywiSF4wKEYpIl0sWzIsMywiSF4wKEUpXFxvdGltZXMgXFxPT197XFxQUF8zfSgtMSkiXSxbMSwzLCJIXjAoXFxPT19TKVxcb3RpbWVzXFxPT197XFxQUF8zfSgtMykiXSxbMSwxLCJ7SF4yKFxcT09fUylcXG90aW1lc1xcT09fe1xcUFBfM30oLTMpfSJdLFswLDNdLFs0LDNdLFszLDRdLFszLDBdLFsxLDIsIi4iXSxbMiwyLCIuIl0sWzIsMSwiLiJdLFszLDIsIi4iXSxbMywxLCIuIl0sWzQsMiwiIiwyLHsic3R5bGUiOnsiaGVhZCI6eyJuYW1lIjoibm9uZSJ9fX1dLFsyLDEsIiIsMix7InN0eWxlIjp7ImhlYWQiOnsibmFtZSI6Im5vbmUifX19XSxbMSwwLCIiLDIseyJzdHlsZSI6eyJoZWFkIjp7Im5hbWUiOiJub25lIn19fV0sWzAsNSwicCJdLFs2LDAsIiIsMCx7InN0eWxlIjp7ImhlYWQiOnsibmFtZSI6Im5vbmUifX19XSxbMCwxMSwiIiwwLHsic3R5bGUiOnsiaGVhZCI6eyJuYW1lIjoibm9uZSJ9fX1dLFsxMSwxMiwiIiwwLHsic3R5bGUiOnsiaGVhZCI6eyJuYW1lIjoibm9uZSJ9fX1dLFsxMiw3LCJxIiwyXV0=
  \[
    \begin{tikzcd}[row sep=small]
	&&& {} \\
	& {{H^2(\OO_S)\otimes\OO_{\PP_3}(-3)}} & {.} & {.} \\
	& {.} & {.} & {.} \\
	{} & {H^0(\OO_S)\otimes\OO_{\PP_3}(-3)} & {H^0(E)\otimes \OO_{\PP_3}(-1)} & {H^0(\mF)} & {} \\
	&&& {}
	\arrow[no head, from=4-1, to=4-2]
	\arrow[no head, from=4-2, to=4-3]
	\arrow[no head, from=4-3, to=4-4]
	\arrow["p", from=4-4, to=4-5]
	\arrow[no head, from=5-4, to=4-4]
	\arrow[no head, from=4-4, to=3-4]
	\arrow[no head, from=3-4, to=2-4]
	\arrow["q"', from=2-4, to=1-4]
      \end{tikzcd} .
    \]

% https://q.uiver.app/#q=WzAsMTAsWzMsMCwiIFdfezEwfSBcXG90aW1lcyBcXE9PX3tcXFBQXzN9ICJdLFswLDAsIjAiXSxbMiwwLCJcXGZyYWN7U18yXFxWdn17TX0gXFxvdGltZXMgXFxPT197XFxQUF8zfSgtMSkiXSxbMSwwLCJcXE9PX3tcXFBQXzN9KC0zKSAiXSxbNiwwLCIgXFxPT197XFxQUF8zfSgtMykiXSxbNywwLCIwIl0sWzQsMSwiVCJdLFszLDIsIjAiXSxbNSwyLCIwIl0sWzUsMCwiKHBfezJ9KV8qIChcXGNMICkiXSxbMiwwXSxbMywyXSxbMSwzXSxbNCw1XSxbMCw2XSxbNiw4LCIiLDAseyJzaG9ydGVuIjp7InRhcmdldCI6MjB9fV0sWzcsNiwiIiwwLHsic2hvcnRlbiI6eyJzb3VyY2UiOjIwfX1dLFs5LDRdLFswLDldLFs2LDldXQ==
% \[\begin{tikzcd}[column sep=small,row sep=tiny]
% 	0 & {\OO_{\PP_3}(-3) } & {\frac{S_2\Vv}{M} \otimes \OO_{\PP_3}(-1)} & { W_{10} \otimes \OO_{\PP_3} } && {(p_{2})_* (\cL )} & { \OO_{\PP_3}(-3)} & 0 \\
% 	&&&& T \\
% 	&&& 0 && 0
% 	\arrow[from=1-3, to=1-4]
% 	\arrow[from=1-2, to=1-3]
% 	\arrow[from=1-1, to=1-2]
% 	\arrow[from=1-7, to=1-8]
% 	\arrow[from=1-4, to=2-5]
% 	\arrow[shorten >=12pt, from=2-5, to=3-6]
% 	\arrow[shorten <=12pt, from=3-4, to=2-5]
% 	\arrow[from=1-6, to=1-7]
% 	\arrow[from=1-4, to=1-6]
% 	\arrow[from=2-5, to=1-6]
% \end{tikzcd}\]
So it ends with only two non zero terms $E_2^{0,0}=T$ and
$E_2^{-2,2}=\OO_{\PP_3}(-3)$. It gives the two exact sequences of the lemma and
the vanishing of $R^{i}(p_2)_*(\cL)$ for $i>0$. The description of $\alpha$ is straightforward
from the exact sequences (\ref{seqsectionE}) and (\ref{seqsectionF}) and
study of $T$ is done in Lemma~\ref{lemmeT}.
\end{proof}

%%%%
%% Definition de beta et des 4 quadriques associées à un plan.
\begin{prop}\label{refpropbeta}
  Let $\beta$ be the following map
  \[
    \frac{S_2\Vv}{M} \otimes \OO_{\PW}(-1) \xrightarrow{\quad\beta\quad} V
    \otimes \OO_{\PW}
  \]
  defined by the element $\alpha \in \Hom(
  \frac{S_2\Vv}{M}, W_{10} \otimes V)$ of Lemma~\ref{refdefalpha}. Let $p_\pi: V \twoheadrightarrow V_\pi$ be a general
  quotient of rank $3$ of $V$ (i.e. the choice of a plane $\pi\subset \PP_3$),
  then the saturation of the ideal of $(3\times 3)$-minors of $p_\pi \circ \beta$ contains four independent elements of $H^0(\IX(2))$.
\end{prop}
\begin{proof}
  Let $p_S$ be the projection from $\PF$ to $S$ and still denote by  $H$ and $h$
  their hyperplane classes (cf. Definition~\ref{refdefHh}). Denote by $C$ the scheme defined in
  $\PW$ by the $(3\times 3)$-minors of $\beta$.

  The vector bundle $E$ is globally generated with $H^0(E)=\frac{S_2\Vv}{M}$, so the
  pullback of $\beta$ on $\PF$ factors through the map $\beta'$
    % https://q.uiver.app/#q=WzAsNSxbMCwwLCIwIl0sWzEsMCwiRSgtSCkiXSxbMywwLCJWXFxvdGltZXNcXE9PX3tcXFBGfSJdLFs0LDAsIlxcY0lfe0N8WH0oM0gtaCkiXSxbNSwwLCIwIl0sWzAsMV0sWzEsMiwiXFxiZXRhJyJdLFsyLDNdLFszLDRdXQ==
  \begin{equation}
    \label{seqbetaprime}
        \begin{tikzcd}[column sep=small]
	0 & {p_S^*E(-H)} && {V\otimes\OO_{\PF}} & {\cI_{C|\PF}(3H-h)} & 0
	\arrow[from=1-1, to=1-2]
	\arrow["{\beta'}", from=1-2, to=1-4]
	\arrow[from=1-4, to=1-5]
	\arrow[from=1-5, to=1-6]
      \end{tikzcd},
  \end{equation}
  where $\cI_{C|\PF}$ is here just an ideal of $\PF$ defined by $\beta'$. But
  a-posteri $C$ is included in $X$  so after Corollary~\ref{refXsmooth}
  $\cI_{C|\PF}$ turns out to be the ideal of $C$ in $\PF$.

  Remark from the definition of
  $T$ in sequence~(\ref{seqresolT}) that $\PP(T^\dual)$ is defined in
  $\PP_9\times \PP_3$ by the vanishing of $p\circ \beta$ where $p:V\otimes
  \OO_{\PP_3} \to \OO_{\PP_3}(1)$ is the tautological quotient map.
Let $T_{\vert \pi}$ be restriction to the plane $\pi$ of the vector bundle
  defined in Lemma~\ref{refdefalpha}. Denote by $\Sigma_\pi$ the image of
  $\PP(T_{\vert \pi})$ in $\PW$. It is the subscheme of $\PW$ defined by the $3\times 3$ minors of
  $p_\pi \circ \beta$. First show that $H^0(\cI_{\Sigma_\pi}(2)) \subset H^0(\cI_{X}(2))$.

  The pullback of $\Sigma_\pi\cap X$ to $\PF$ is defined by
  the determinant of $p_\pi\circ \beta'$ so its class is $3H-h$. From
  $h^0(\OO_{\PF}(2H-(3H-h))=0$, we get that any quadric containing $\Sigma_\pi$
  contains also $X$.

  It remains to prove that $h^0(\cI_{\Sigma_\pi}(2))\geq 4$. From
  Lemma~\ref{lemmeT} we have  ${h^0(S_2 (T_{\vert \pi})) = 51}$ so
  $h^0(\cI_{\Sigma_\pi}(2)) \geq h^0(\OO_{\PW}(2))-h^0(S_2 (T_{\vert \pi})) =4$,
  and the proposition is proved.

\smallskip
Moreover  $C$ is a subscheme of $\Sigma_\pi$ so $H^0(\cI_C(2))$ contains  $\bigcup_{\pi \in \PP_3^\dual}
H^0(\cI_{\Sigma_\pi}(2))$.
But Theorem~\ref{refequationsX} shows that this union spans $H^0(\cI_X(2))$ and
that $\cI_X$ is generated by quadrics so $C$ is a subscheme of $X$.
\end{proof}
\begin{lemm}\label{lemmeT}
  The cokernel $T$ of $\alpha$ is a vector bundle of rank $3$ with
  \begin{itemize}
  \item[---] Chern classes $c_1(T)=5$, $c_2(T)=12$, $c_3(T)=12$ and
  \item[---]  $H^0(T)=W_{10}$, $i>0~\Rightarrow~h^i(T)=0$
  \item[---] $\chi(S_2(T))=53$,  $\chi((S_2(T))(-1))=2$.
  \item[---] For any plane $\pi$ of $\PP_3$  we have $\chi(S_2(T_{\vert
      \pi}))=51$, $i>0 \Rightarrow h^i(T_{\vert\pi} \otimes T_{\vert\pi}) = 0$.
  \end{itemize}
\end{lemm}

\begin{proof}First remark that for $M$ general, the tautological inclusion
  $$0 \rightarrow \OO_{\PP_3}(-3) \rightarrow S_2\Vv \otimes \OO_{\PP_3}(-1)$$ is still
  injective after composition with the projection $$S_2\Vv \otimes
  \OO_{\PP_3}(-1) \rightarrow \frac{S_2\Vv}{M} \otimes \OO_{\PP_3}(-1).$$
  So the kernel of $\alpha$ contains $\OO_{\PP_3}(-3)$ and the sheaf $T$ has rank at
  least $3$.
  But we have checked on an example with \cite[Macaulay2]{macaulay2} that the saturation of
  the ideal generated by the $3\times3$-minors of $\alpha$ is $(1)$, so $T$ is a
  vector bundle of rank $3$.

The exact sequence~(\ref{seqresolT}) gives $H^0(T)=W_{10}$, and the vanishing
$h^i(T)=0$ for $i>0$. Its restriction to a plane $\pi$ gives
$h^i(T_{\vert\pi})=h^i(T_{\vert\pi}(-1))=0$ for $i>0$. So for $i>0$ we have
$h^i(T_{\vert\pi} \otimes T_{\vert\pi}) = 0$.
\end{proof}

  \begin{rema}\nbmarge{C'est peut être une piste intéréssante?}
    Note that $h^0(S^2(T))(-1)=2$ marks a pencil of quadrics containing $X$. So the choice
    of $E$ gives a line in this space of quadrics $\PV$. We should understand more this choice. Note
    that J.Meng  made recent progress in this direction. The
    moduli space of $E$ is called the Fourier-Mukai partner $(S',h')$ of $(S,h)$
    and he proved (cf. \cite[Lemma~4.1]{meng}) that $\PP(T_{S'}^1(-h'))$ has a natural projection to $\PV$.

    This pencil can be computed from Proposition~\ref{refpropbeta} as the intersection of several
    $H^0(\cI_{\Sigma_\pi}(2))$ for generic choices of planes $\pi$ in
    $\PP_3$. Computing this pencil on random examples always gives us a pencil
    of quadrics of rank $8$ (so included in the discriminant hypersurface). Moreover it was
    always a $4$-secant line to the Peskine\footnote{cf. §4.} variety $Y^1_{t_1}\subset \PV$.
  \end{rema}

  \begin{rema}
   For any point $x$ of $\PP_3$ denote by $\widetilde{\cC_S}_x$  the fiber of
   $\widetilde{\cC_S}$ (cf. Remark~\ref{refCStilde}) over $x$. Then
   $\widetilde{\cC_S}_x$ is the primage in $\PF$ of the intersection of $X$ with
   the plane $\PP(T^\dual_x)$.
%% courbe univ
% https://q.uiver.app/#q=WzAsNyxbMSwwLCJcXFBGIl0sWzQsMSwiXFxQUF8zIl0sWzAsMSwiUyJdLFszLDEsIlxcUFAoVF5cXGR1YWwpIl0sWzMsMCwiXFxjQ19TIl0sWzIsMSwiXFxQUF85Il0sWzEsMSwiWCJdLFswLDIsInBfUyIsMl0sWzQsMF0sWzQsMywiIiwyLHsic3R5bGUiOnsidGFpbCI6eyJuYW1lIjoiaG9vayIsInNpZGUiOiJ0b3AifX19XSxbMywxXSxbNCwxLCJ7XFx0aW55IDQ6MX0iXSxbMyw1XSxbMCw2XSxbNiw1LCIiLDAseyJzdHlsZSI6eyJ0YWlsIjp7Im5hbWUiOiJob29rIiwic2lkZSI6InRvcCJ9fX1dXQ==
\[\begin{tikzcd}[column sep=scriptsize,row sep=small]
	& \PF && {\widetilde{\cC_S}} \\
	S & X & {\PP_9} & {\PP(T^\dual)} & {\PP_3}
	\arrow["{p_S}"', from=1-2, to=2-1]
	\arrow[from=1-4, to=1-2]
	\arrow[hook, from=1-4, to=2-4]
	\arrow[from=2-4, to=2-5]
	\arrow["{{\tiny 4:1}}", from=1-4, to=2-5]
	\arrow[from=2-4, to=2-3]
	\arrow[from=1-2, to=2-2]
	\arrow[hook, from=2-2, to=2-3]
\end{tikzcd}\]
\end{rema}
\begin{proof}
  Let $p_x$ be the  quotient of $V$ defined by $x$. The fiber $\widetilde{\cC_S}_x$
  is defined in $\PF$ by the vanishing of $p_x \circ \beta'$ (cf. sequence~(\ref{seqbetaprime})). But the
  vanishing of $p_x \circ \beta$ defines in $\PW$ the plane $\PP(T^\dual_x)$. So
  $\widetilde{\cC_S}_x$ is just defined by the intersection $X\cap
  \PP(T^\dual_x)$ and $\PP(T^\dual_x)$ is a family of quadrisecant planes to $X$.

\end{proof}

\subsection{Equations of $X$ in $\PP_9$}
\begin{claim}
The Hilbert polynomial of $X$ is \[\chi(\OO_{\PF}(mH))=P_X(m)=21P_3(m)-36P_2(m)+17P_1(m)=\frac{7
    m^{3}+6 m^{2}+13 m+14}{2},\]
  where $P_i$ is the Hilbert polynomial of the projective space of dimension
  $i$.
%--o56 = 17*P  - 36*P  + 21*P
%--          1       2       3
\end{claim}
%%\nbmarge{faut il rappeler ci dessous que l'on est dans le cas generique pour S?}
Note that in  \cite[Lemma~8.3]{DHOV} we only proved $h^0(\cI_X(2)) \geq 10$ but
an example with equality was still missing. In the following theorem we are able
to prove the equality and compute these equations for a generic $S$.
\begin{theo}\label{refequationsX}
We have $h^0(\cI_X(2)) = 10$ and the ideal of $X$ in $\PP_9$ is generated by
these quadrics. Moreover these quadrics are obtained from
Proposition~\ref{refpropbeta} by chosing different values of $p_\pi$.
\end{theo}
\begin{proof}
From \cite[Lemma~8.3]{DHOV} that we always have $h^0(\cI_X(2))\geq
10$ so it is enough to prove the equality on an example to obtain it for the
generic $S$.

Take $\beta$ from a general  example as in Proposition~\ref{refpropbeta}. For each
triple of rows of $\beta$ we obtain $4$ quadrics containing $X$. Doing this on
an example with \cite[Macaulay2]{macaulay2} we obtain only $10$ linearly
independent quadrics. Now consider the scheme $X'$ defined by these
quadrics and compute its Hilbert polynomial with Macaulay2.  We found that $X'$
ands $X$ have same Hilbert polynomial, and by Proposition~\ref{refpropbeta}, the
scheme $X'$
contains $X$. So there is no residual scheme and $X$ is defined by these $10$
quadrics.
\end{proof}

\subsection{A canonical involution $\iX$ on $\PP_9$}
As a corollary of Theorem~\ref{refequationsX} we obtain a natural involution on
$\PP_9$ with indeterminacy locus $X$, it depends only of the polarized
K3-surface $(S,h)$, and not of the choices made in Mukai's unirationalization.
\begin{coro}\label{refdoublecover}
  Let $V_{10}=H^0(\cI_X(2))$
The linear system $|\cI_X(2)|$ defines a morphism of degree $2$ from the blowup
of $\PP_9$ in $X$ to the $9$-dimensional projective space $\PVv$.
\[
|\cI_X(2)|\colon~ \widetilde{\PP_9(X)} \xrightarrow{\quad 2:1 \quad} \PVv
\]
In the following denote by $\iX$ the rational involution on $\PW$ defined by
this morphism and by $\fX$ the quotient map.
\end{coro}
\begin{proof}
From Theorem~\ref{refequationsX} the base scheme of $|\cI_X(2)|$ is $X$. So it is
just a class computation in the Chow ring of $\PF$. The Segre class of the
normal bundle $N$ is
\[
s(N)=1 + (- 8H  - h) + (45h\cdot H  - 291p) - 4152p\cdot H
\]
where $p$ is the class of a point on a rational curve of $S$
and we have
\[
  s(N)\cdot (1+2H)^9 = 1 + (10H  - h) + (27h\cdot H  - 291p) + 510p\cdot H.
\]
So the projection from $\widetilde{\PP_9(X)}$ to $\PVv$ has degree $2^9-510 =2$.
\end{proof}

\section{Canonical constructions from $(S,h)$}
\subsection{Equations of $S$ in $\PP_{16}$}
  The vector bundle $F$ maps  $S$ into the Grassmannian
  $G(2,W_{10}^\dual)$ with Plücker polarization the very ample line bundle
  $\wedge^2 \mF = \cO_S(h)$.
  Denote by $\PP_{16}$ the sixteen dimensional projective
  space spanned by $S$ in the Plücker embedding.

  We have now for a general
  complex K3-surface of genus sixteen the following
\begin{theo}\label{refequationsS}
  The ideal of $S$ in $\PP_{16}$ is generated by the restriction of the Plücker
  quadrics of the Grassmannian, and  the ideal of $S$ in the Grassmannian $G(2,10)$
  is generated by the $28$ linear equations of this sixteen dimensional
  projective space. In other words, the intersection $\PP_{16}\cap G(2,10)$ has
  unexpected dimension but it still defines (as scheme) the surface $S$.
\end{theo}
\begin{proof}
  Following the construction of §2 we chose an example over $\FF_{101}$. We were able with \cite[Macaulay2]{macaulay2} to get enough lines in $X$ to span this
  $\PP_{16}$. Hence we got a skew symmetric map from $W_{10}^\dual\otimes\OO_{\PP_{16}}(-1) \to
  W_{10}\otimes \OO_{\PP_{16}}$ and let $S'$ be the scheme defined by its
  $(4\times4)$-paffians. So $S$ is included in $S'$ but we were able to compute the Hilbert polynomial of
  $S'$ and it was equal to the Hilbert polynomial of $S$. So the residual scheme
  is empty and $S=S'$. So this residual scheme is also empty for a general
  complex K3-surface of genus sixteen.
\end{proof}
\begin{rema}\label{refbettitable}
  On a finite field one can compute the resolution of $X$ in $\PP_9$ and find
  directly the previous equations from the last differentials. Over $\mathbb{Q}$
  we were also able to obtain the resolution of $X$ in less than a day on an  example with very
  small\footnote{only $0$ or $1$} coefficents in the initial choices of $M$ and
  $N$. \nbmarge{Mais cet exemple n'a pas une varieté de Peskine pour $t_2$
    lisse, et nous n'avons pas testé la lissité de $X$. De plus un exemple un
    peu plus compliqué avec uniquement des $0,1,-1$ dans les données initiales
    n'a pas reussi à trouver la resolution sur $\mathbb{Q}$. (>50GoRam et plante
    à 4 semaines)}
\end{rema}
\begin{proof}
  \ifdefined\showmacaulay

  It is enough to have an example where this resolution ends like this
\[
\dots \leftarrow \OO_{\PW}^{28}(-8)  \xleftarrow{\  d_6 \ } \OO_{\PW}^{10} (-9) \xleftarrow{\ d_7\ }
\OO_{\PW}(-10) \leftarrow 0.
\]
Note that random examples in \cite[Macaulay2]{macaulay2} gives the following Betti table for the
resolution of $\OO_X$.
\begin{verbatim}
i70 : betti resolX

             0  1  2  3  4  5  6 7
o70 = total: 1 10 43 71 56 28 10 1
          0: 1  .  .  .  .  .  . .
          1: . 10  8  .  .  .  . .
          2: .  . 35 70 36  .  . .
          3: .  .  .  1 20 28 10 1
\end{verbatim}
\fi
So  $d_7$ is just a part of  Euler sequence, and the functor ${\mathcal
  Hom}(~\cdot~, \OO_{\PW}(8))$ gives the following exact sequence
\[
\OO_{\PW}^{28}\xrightarrow{\ ^td_6\ } \Omega^1_{\PW}(2) \longrightarrow
{\mathcal Ext}^6(\OO_X,\OO_{\PW}(-8)) \to 0.
\]
%Note that ${\mathcal Ext}^6(\OO_X,\OO_{\PW}(-8))=p_S^*(\OO_S(h))$,  %% a
%posteriori apres X iso a P(F*)
So it shows
that $\PF$ is defined in the incidence variety
$\Proj(Sym(\Omega_{\PW}(2)))$ by these $28$ linear sections of the Grassmannian.
\end{proof}
\subsection{Smoothness of $X$}
\begin{coro}\label{refXsmooth}
  Still assume that $(S,h)$ is a general K3-surface of genus sixteen, then the projection $\PF \to X$ is an isomorphism.
\end{coro}
\begin{proof}
  Note that $F$ is simple (\cite[\S7 claim~5]{mukaigenus16}), so $X$ can't be a
  cone. From the general assumption on $(S,h)$, the  Picard group of $S$ has
  rank $1$ so $F$ is ample\footnote{Thanks to L.Manivel for pointing this
    reference to me.} by a result of Beauville (cf. \cite{beauvilleample}). \nbmarge{Sur quel
    corps \c ca reste vrai?} So the projection
  $\PF \to X$ is finite.

  Now remark from Theorem~\ref{refequationsS} that $\PF$ is
  a subscheme of the point/line incidence   variety
  $\Proj(Sym(\Omega^1_{\PP_9}(2)))$ defined by linear equations in this relative
  projective   space. So the projection $\PF \to X$  must be an isomorphism because all its fibers are finite schemes defined by linear equations.
\end{proof}
%%%%%%%%%%%%%%%%%%%%%%%%%%%%%%%%%%%%%%%%%%%%%%%%%%%%%%%%%%%%%%%%%%%%%%%%%%%%%%%%%%%%%%%%%%%%%%%%
\subsection{Basic pieces}
\begin{nota}
  Let's consider the following data canonically defined by $(S,h)$.
  \begin{itemize}[label=---]
  \item Let $\NXv$ be the conormal bundle of $X$ in $\PW$.
  \item Denote by $\cP$ the bundle of principal parts of $X\subset\PW$ sitting
    in the following extension
    \[
      0 \to \Omega^1_X(H) \to \cP \to \OO_X(H) \to 0.
    \]
  \item $\VV=H^0(\cI_X(2))$ the $10$-dimensional vector space of quadrics
    containing $X$. Denote by $\gamma$ the following inclusion and $J_\gamma$
    the Jacobian map
    \[
      \gamma\colon \VV \subset S^2(\Wdix),\quad
      \VV \otimes \OO_{\PW}(-1)      \xrightarrow{\ J_{\gamma}\ } \Wdix \otimes \OO_{\PW}.
    \]
  \end{itemize}
\end{nota}
\begin{lemm}\label{reflemmaimagesdirectes}
  We have the equalities
  \begin{itemize}[label=---]
  \item $\Omega^1_S = p_{S*}(\Omega^1_X) = R^1p_{S*}(\NXv)$, $\omega_{p_S}=\OO_X(h-2H)$.
  \item The direct image $\mG := p_{S*}(\NXv(H))(h)$ is a vector bundle of rank
    $4$ on $S$ with the same\footnote{cf. Corollary~\ref{refG4mukai} for the
      stronger result that these two vector bundle are isomorphic.}\footnote{See
    also \cite[\S7.3]{kuzderivedfamilyfano} for another definition of this
    bundle on a surface with higher Picard rank.}  Chern classes as the Lazarsfeld-Mukai bundle of
    rank $4$
    \[
      c_1(\mG)=h,\ c_2(\mG)=15,\ h^0(\mG)=8,\ \chi(\sHom(\mG,\mG))=2.
    \]

  \item We have the following canonical identifications
    \[
      p_{S*}(\NXv(2H))=\VV\otimes \OO_S
    \]
  \end{itemize}
\end{lemm}
\begin{proof}
  From Theorem~\ref{refequationsX}  the vector bundle $\NXv(H)$ is a quotient of
  $\VV\otimes \OO_X(-H)$, so $R^1p_{S*}(\NXv(H))=0$ and the fiber of $\mG$ at
  any point $s$ of $S$ is $H^0(\NXv(H)_{|p_S^{-1}(s) })$. But
  $\NXv(H)_{|p_S^{-1}(s) }$ is both a
  quotient of  $\VV\otimes \OO_{\PP_1}(-1)$ and a subsheaf of the restriction of
  $\Omega^1_{\PW}(1)$ to the line $p_S^{-1}(s)$. Hence we have for any $s$ in $S$
  the isomorphism
  $$
\NXv(H)_{|p_S^{-1}(s) } \simeq \OO_{\PP_1}^4 \oplus \OO_{\PP_1}^2(-1)
  $$
and $\mG$ is locally free of rank $4$. We have also $R^1p_{S*}(\NXv(2H))=0$, so we
  can show  that the first Chern class of $p_{S*}(\NXv(2H))$ is $0$ from
  Grothendieck-Riemann-Roch. Hence  $p_{S*}(\NXv(2H))$ is isomorphic to $\VV\otimes \OO_S$.

  NB: we have $\chi(\mG)=8$, so it is an open condition to have $h^0(\mG)=8$. We
  checked this equality on an example in Corollary~\ref{refG4mukai}.
\end{proof}

\begin{lemm}[Vertex bundle]\label{refvertexbundle}
  The restriction of the Jacobian map $J_{\gamma}$ to $X$ gives the
     following exact sequences
     \begin{equation}\label{eqvertexbundle}
         0 \to p_S^*(\mG )(-2H) \to \VV \otimes \OO_{X}(-H) \to \NXv(H) \to 0.
     \end{equation}
     So for a point $x$ of $X$, the fiber $(p_S^*(\mG ))_{x}\subset \VV$ is the locus of quadrics containing
     $X$ and singular at $x$.
   \end{lemm}
   \begin{proof}
     First denote by $N_4$ the kernel of the right side map in this sequence. We have to show that
     $N_4\simeq p_S^*(\mG(-2H))$.

     By definition of $\mG$ in Lemma~\ref{reflemmaimagesdirectes} we have
     $R^1p_{S*}(N_4) = \mG(-h)$. So
     the relative Beilinson spectral sequence
     \begin{equation}
       \label{eqrelativebeilinson}
       E_1^{-p,q} (\cF)= p_S^*(R^qp_{S*}(\cF(-pH)))\otimes \omega_{P_S}(pH)
       \Rightarrow \cF~ \mbox{ (in degree $0$)}
     \end{equation}
       gives for $\cF=N_4(H)$ the following exact sequence
     \[
       p_S^*\left(p_{S*}(N_4(H))\right) \to N_4(H) \to \omega_{p_S}(H) \otimes p_S^*(\mG(-h))\to p_S^*\left(R^1p_{S*}(N_4(H))\right).
     \]
     But we have $0=p_{S*}(N_4(H))= R^1p_{S*}(N_4(H))$  because they are the
     kernel and cokernel of the isomorphism $\VV \otimes\OO_S \to
     p_{S*}(\NXv(2H))))$. So $N_4$ is isomorphic to $p_S^*(\mG(-2H))$.
   \end{proof}

   \begin{coro}\label{refOmegaS}
      From this relative construction we get
 the exact sequence
    \[ 0 \to \mG \otimes \mF^\dual \to \VV\otimes \OO_S \to \Omega^1_S(h )\to 0.
    \]
  and a natural identification \[\VV=H^0(\Omega^1_S(h)).\]
\end{coro}
\begin{proof}
  From spectral sequence~(\ref{eqrelativebeilinson}) with $\cF=\NXv(H)$ we get
  with Lemma~\ref{reflemmaimagesdirectes}   the following exact sequence
  \[
      0 \to p_S^*(\mG(-h))) \to \NXv(H) \to \omega_{p_S}(H)\otimes
      p_S^*(\Omega^1_S)\to 0.
  \]
  It gives the exact sequence after tensorization by $\OO_X(H)$ and push
  forward. Defining in \cite[Macaulay2]{macaulay2} $p_S^*\mGdual$ as the sheaf
  on $X$  image of the map ${}^t s_\gamma$ (cf. Corollary~\ref{refG4mukai}) we
  computed\footnote{(took a bit less than 5h each)}
  $h^2(p_S^*\mGdual(H))=0$ and $h^1(p_S^*\mGdual(H))=0$. So for $S$ generic, we have the
  following vanishing $h^0(\mG \otimes
  \mF^\dual )=h^1(\mG \otimes \mF^\dual)=0$ and  $\VV=H^0(\Omega^1_S(h))$.
\end{proof}
\begin{lemm}[multiple structure bundle]
  The following exact sequence
  \[
    0 \to \mG(-h) \to \Wdix \otimes \OO_S \to p_{S*}(\cP) \to 0
  \]
  describes the fiber $\mG_s$ over a point $s$ of $S$ as the equations of a
  linear space spanned in $\PW$ by a multiple structure on the line $p_S^{-1}(s)$.
\end{lemm}
\begin{proof}
  Note that $\cP$ is the bundle of principal parts of $X$ so we have the exact sequence
  \[
    0\to \NXv(H) \to \Wdix \otimes \OO_X \to \cP \to 0,
  \]
  and the Lemma follows after applying $p_{S*}$ to it.
\end{proof}
%%%%%%%%%%%%%%%%%%%%%%%%%%%%%%%%%%%%%%%%%%%%%%%%%%%%%%%%%%%%%%%%%%%%%%%%%%%%%%%%%%%%%%%%%%%%%%%%%
\subsection{Syzygies of $X$ and the trivector $t_2$}
\begin{prop}\label{refsgamma}
  The eight dimensional vector space $V_8=H^0(\mG)$ is the space of linear
  syzygies of the quadrics containing $X$.

  Denote by $\sigma_\gamma \colon ~ V_8 \xhookrightarrow{\quad   \sigma_\gamma\quad } \VV \otimes \Wdix$ the
  associated inclusion and by $s_\gamma$
  \[
    V_8 \otimes \OO_{\PW} \xrightarrow{\ s_\gamma \ } \VV\otimes \OO_{\PW}(1).
  \]
  Moreover, the image of the restriction of $s_\gamma$ to $X$ is $p_S^*(\mG)$.
\end{prop}
\begin{proof}
  Denote by $V_8$ the linear syzygies of the quadrics containing $X$. From
  sequence~(\ref{eqvertexbundle}), the image of the restriction of
  $s_\gamma$ to $X$ is a subsheaf of $p_S^*(\mG)$. So for any point $s$ of $S$, the restriction of
  $s_\gamma$ to the line $p_S^{-1}({s})$ factors through
  $$
V_8 \otimes \OO_{\PP_1} \rightarrow \OO_{\PP_1}\otimes \mG_s = \OO_{\PP_1}^4 .
$$
But such a map is either surjective or of rank at most $3$. As we were able to
check on an example that the $4\times 4$-minors of $s_\gamma$
defines the emptyset on $X$, the image of the restriction of $s_\gamma$ to $X$
is $p_S^*(\mG)$. So  we have $V_8=H^0(\mG)$ and the vector bundle $\mG$ is globally generated.
\end{proof}
\begin{prop}\label{refphiV8}
  For $S$ general, there is  skew-symmetric
  isomorphism $$\phi: V_8^\dual \to V_8$$ such that the vector bundle $\mGdual$ is
  isotropic for $\phi$. This isomorphism is  uniquely defined up to a constant
  factor and can be computed as the kernel of $s_\gamma \otimes s_\gamma$.
\end{prop}
\begin{proof}
  First remark that $\chi(\wedge^2 \mG)=27$ so it is an open condition to have
  $h^0(\wedge^2 \mG)=27$. From Lemma~\ref{refvertexbundle} we can define the
  sheaf $p_S^*(\mG)$  as the image of $s_\gamma$. So we were able to
  compute\footnote{it took between 24h/48h} this equality on an example
  in \cite[Macaulay2]{macaulay2}. So for $S$ general there is  a skew-symmetric map $\phi:
  V_8^\dual \to V_8$ such that $\mGdual$ is isotropic for $\phi$. But on the
  previous example we computed the kernel of $s_\gamma \otimes s_\gamma$, and
  found only a one dimensional space spaned by a skew-symmetric isomorphism. So
  for $S$ general, $\phi$ is an isomorphism.
\end{proof}
% X=variety(IX)
% M810X=substitute(M810,ring X);
% G4=(sheaf(X,image(M810X)))(2);
% G4
% --HH^0(G4(1)) -- OK 8
% i190 : G4=(sheaf(X,image(substitute(M810,ring X))))(3)

%                                            10
% o190 : coherent sheaf on X, subsheaf of OO   (1)
%                                           X

% o192 = k

% o192 : GaloisField

% i193 : time HH^0(G4)
%      -- used 4180.09 seconds

%         8
% o193 = k

% o193 : k-module, free
%%%%%%%%%%
%%%%%%%%%
%
% W2=schur({1,1},M810);
% W2X=substitute(W2,ring X);
% W2G4=(sheaf(X,image W2X))(6);
% --24/48h-- HH^0(W2G4) -- OK 27 (un jour ou deux?)
%
%
% -- i40 : loadPackage("SchurFunctors");
% -- --warning: symbol "straighten" in PieriMaps.Dictionary is shadowed by a symbol in SchurFunctors.Dictionary
% -- --  use the synonym PieriMaps$straighten

% -- i41 : W2=schur({1,1},M810);

% --               45       28
% -- o41 : Matrix R   <--- R

% -- i42 : W2X=substitute(W2,ring X);

% --                45        28
% -- o42 : Matrix RX   <--- RX

% -- i43 : W2G4=(sheaf(X,image W2X))(6);

% -- i44 : HH^0(W2G4)

% --        27
% -- o44 = k

% -- o44 : k-module, free

Note that we have from the previous sequences:
\begin{coro}\label{refG4mukai}
  The restriction of $s_\gamma$ and $J_\gamma$ to $X$ gives the following exact sequences
% https://q.uiver.app/#q=WzAsOSxbMiwwLCJWXzhcXG90aW1lcyBcXE9PX1giXSxbNCwwLCJcXFZWXFxvdGltZXMgXFxPT19YKEgpIl0sWzYsMCwiXFxXZGl4IFxcb3RpbWVzIFxcT09fWCgySCkiXSxbNywwLCJcXGNQKDJIKSJdLFs4LDAsIjAiXSxbNSwxLCJcXE5YdigzSCkiXSxbMywxLCJwX1NeKihcXG1HKSJdLFsxLDAsInBfU14qKFxcbUdeXFxkdWFsKSJdLFswLDAsIjAiXSxbMyw0XSxbMSwyLCJKX1xcZ2FtbWEiXSxbMCwxLCJzX1xcZ2FtbWEiXSxbMiwzXSxbNSwyLCIiLDAseyJzdHlsZSI6eyJ0YWlsIjp7Im5hbWUiOiJob29rIiwic2lkZSI6InRvcCJ9fX1dLFsxLDUsIiIsMCx7InN0eWxlIjp7ImhlYWQiOnsibmFtZSI6ImVwaSJ9fX1dLFswLDYsIiIsMix7InN0eWxlIjp7ImhlYWQiOnsibmFtZSI6ImVwaSJ9fX1dLFs2LDEsIiIsMix7InN0eWxlIjp7InRhaWwiOnsibmFtZSI6Imhvb2siLCJzaWRlIjoidG9wIn19fV0sWzgsN10sWzcsMF1d
\[\begin{tikzcd}[cramped,sep=tiny]
	0 & {p_S^*(\mGdual)} & {V_8\otimes \OO_X} && {\VV\otimes \OO_X(H)} && {\Wdix \otimes \OO_X(2H)} & {\cP(2H)} & 0 \\
	&&& {p_S^*(\mG)} && {\NXv(3H)}
	\arrow[from=1-8, to=1-9]
	\arrow["{J_\gamma}", from=1-5, to=1-7]
	\arrow["{s_\gamma}", from=1-3, to=1-5]
	\arrow[from=1-7, to=1-8]
	\arrow[hook, from=2-6, to=1-7]
	\arrow[two heads, from=1-5, to=2-6]
	\arrow[two heads, from=1-3, to=2-4]
	\arrow[hook, from=2-4, to=1-5]
	\arrow[from=1-1, to=1-2]
	\arrow[from=1-2, to=1-3]
      \end{tikzcd}\]
   For $S$ general the vector bundle $\mG$ is globally generated and stable, so
   it is isomorphic to the Lazarsfeld-Mukai bundle of rank $4$. \end{coro}
 \begin{proof}
   We were able to compute with \cite[Macaulay2]{macaulay2} an example defined
   over $\mathbb{F}_{101}$ and check that the ideal
   generated by the $4\times 4$ minors of $s\gamma$ defines the empty set on
   $X$. So the cokernel of ${}^t s_{\gamma}$ is locally free of rank $4$ on $X$
   hence it is $p_S^*(\mG)$. Hence we  defined $p_S^*(\mGdual)$ as the image
   of the restriction of ${}^t s_{\gamma}$ to $X$ and compute $h^0(\mGdual)=0$ on this example.
   So we have $h^0(\mGdual)=0$ for a general $S$ with $Pic(S)=h\cdot \ZZ$.

   Now consider a quotient $B$ of rank $r_B$ with $1\leq r_B\leq 3$ of $\mG$. It is
   also globally generated so  $r_B-1$ sections of $B$ give a quotient of
   rank $1$ of $\mG$ with first Chern class $c_1(B)=a\cdot h$. But
   $h^0(\mGdual)=0$ implies $a>0$ hence we can't have $\frac{c_1(B)\cdot
     h}{r_B}<\frac{h^2}{4}$ and $\mG$ is stable. From
   Lemma~\ref{reflemmaimagesdirectes} the vector bundle $\mG$ must be isomorphic to the
   Lazarsfeld-Mukai bundle of rank $4$ on $S$.
 \end{proof}

\begin{defi}[Congruence of lines]\label{refcongruenceduntriv}
A congruence of lines is a subvariety of a Grassmannian of lines $G$ of
dimension $\frac{\dim(G)}{2}$. Its order is the number of lines through a
general point of the projective space.

Recall that a trivector gives a section of the dual of the tautological quotient
bundle of $G$ tensored  by $\OO_G(1)$.  As soon as it has the expected
dimension, its zero locus is a congruence of order
$1$ and defines the trivector up to a non zero constant factor.
\end{defi}

\begin{prop}[Invariant lines and $t_2$]\label{refdeftriv2}
  Let $B_{\iX}\subset G(2,\Wdix^\dual)$ be the closure of the set of
  $\iX$-invariant lines of $\PW$ not intersecting $X$.

  Then $B_{\iX}$ is a congruence of order $1$ and there is a trivector $t_2 \in
  \bigwedge^3 \VV$ such that the order one congruence of lines $\Btriv\subset G(2,\VVv)$ defined
  by $t_2$ (cf. Def~\ref{refcongruenceduntriv})
  consists in the image by $\fX\colon\begin{tikzcd}[column sep=3.em]
	\PW & \PVv
	\arrow["{2\colon 1}", dashed, from=1-1, to=1-2]
\end{tikzcd}$ of lines of  $B_{\iX}$.

  Note that $t_2$ is uniquely defined up to a non zero constant and that the
  $8$-dimentional varieties $B_{\iX}$   and $\Btriv$ are birational.
\end{prop}
\begin{proof}
  Let $p$ be a general point in $\PW$. First remark that there is at most one
  element of $B_{\iX}$ through $p$ because this line contains $p$ and $\iX(p)$.

Let $V_{9,p}$ be the nine dimensional vector space $H^0(\cI_{X\cup p}(2))$ and
$\pi$ the quotient map $\pi\colon~ \VV \to \VV/V_{9,p}$ defining the point
$\fX(p)$ of $\PVv$.

Let $l_p$ be the line defined by the vanishing of $\pi \circ s_{\gamma}$. So the
restriction $s_{\gamma,p}$ of $s_\gamma$ to $p$ maps $V_8$ to $V_{9,p}$ and among the nine
quadrics containing $X$ and $p$ eight contain also the line $l_p$. Hence the
restriction to $l_p$ of the linear system $|\cI_X(2)|$  is just a double cover
of a line $\fX(l_p) \subset \PVv$. So $l_p$ is an $\iX$-invariant line
containing $p$ and we have the following
\begin{equation}
  \label{eq:noyaucongruence}
p\mbox{ general in } \PW \Rightarrow \fX(p) \in \ker(^t s_{\gamma,p})
\end{equation}
Now remind from Proposition~\ref{refphiV8}   that  the
  following composition is zero on $X$
  % https://q.uiver.app/#q=WzAsNCxbMCwwLCJcXFZWdlxcb3RpbWVzIFxcT09fe1h9Il0sWzEsMCwiVl84XlxcZHVhbCBcXG90aW1lcyBcXE9PX3tYfShIKSJdLFsyLDAsIlZfOFxcb3RpbWVzXFxPT197WH0oSCkiXSxbMywwLCJcXFZWXFxvdGltZXNcXE9PX3tYfSgySCkiXSxbMCwxLCJedCBzX1xcZ2FtbWEiXSxbMSwyLCJcXHBoaSJdLFsyLDMsInNfXFxnYW1tYSJdLFswLDMsIjAiLDIseyJjdXJ2ZSI6M31dXQ==
\[\begin{tikzcd}
	{\VVv\otimes \OO_{X}} & {V_8^\dual \otimes \OO_{X}(H)} & {V_8\otimes\OO_{X}(H)} & {\VV\otimes\OO_{X}(2H)}
	\arrow["{^t s_\gamma}", from=1-1, to=1-2]
	\arrow["\phi", from=1-2, to=1-3]
	\arrow["{s_\gamma}", from=1-3, to=1-4]
	\arrow["0"', curve={height=24pt}, from=1-1, to=1-4]
\end{tikzcd}.\]
So the coefficients of the skew-symmetric map $s_\gamma \circ \phi \circ
{}^t s_\gamma$ are elements of $\VV$ because it is defined as
$H^0(\IX(2))$. Hence it defines an element $t_2$ of
$(\bigwedge^2 \VV ) \otimes \VV$.

Now remind from~\ref{eq:noyaucongruence} that for a general point $p$ in $\PW$ the point $\fX(p)$ is general
in $\PVv$ such that
\[
\fX(p) \in \ker(t_{2,\fX(p)}),\quad  t_{2,\fX(p)} \in \Hom(\VVv,\VV),~
{}^t t_{2,\fX(p)} = - t_{2,\fX(p)}  .
\]
So $t_2\in \bigwedge^3 \VV$ and is uniquely defined up to a non zero constant
(cf. Def~\ref{refcongruenceduntriv}).

Let $l\subset \PW$ be a general member of the congruence $B_{\iX}$ and $d\subset
\PVv$ its image by $\fX$. The residual
of $l$ in the curve $\fX^{-1}(d)$ must be contracted by $\fX$, because the
restriction of $\fX$ to $l$ is already of degree $2$ over $d$. Hence the
rational map from $B_{\iX}$ to $\Btriv$ induced by $\fX$ has degree $1$.
\end{proof}
%%%%%%%%%%%%%%%%%%%%%%%%%%%%%%%%%%%%%%%%%%%%%%%%%%%%%%%%%%%%%%%%%%%%%%%%%%%
% Avec les faisceaux c'est long, mais la syntaxe est simple.
% avec k=ZZ/101 on peut tout de meme
%  verifier en moins de 48h le calcul que le h^0 vaut generiquement la
%  caracteristique d'Euler.
%
% -- i40 : loadPackage("SchurFunctors");
% -- i41 : W2=schur({1,1},M810);
% --               45       28
% -- o41 : Matrix R   <--- R
% -- i42 : W2X=substitute(W2,ring X);
% --                45        28
% -- o42 : Matrix RX   <--- RX
% -- i43 : W2G4=(sheaf(X,image W2X))(6);
% -- i44 : HH^0(W2G4)
% --        27
% -- o44 = k
% -- o44 : k-module, free

\nbmarge{NB: $\chi(S^2(\mG))=35$ but $\Ext^1(\mG,\mG^\dual)\neq 0$ so  there
  is no contradiction,
  $\mG$ is not isotropic for a quadratic form.}

%%%%%%%%%%%%%%%%%%%%%%%
%  NB: syzygies quadratiques inutiles.
%%%%%%%%%%%%%%%%%%
% \begin{claim}
%   The injective map $t_2\colon~ \VVv \longrightarrow \bigwedge^2 \VV$ is such
%   that $\coker(t_2)$ defines the quadratic syzygies of $\cI_X$.
% \end{claim}

% \begin{proof}
% Consider the Koszul complex of the $10$ quadrics containing $X$.
% \[
%   \dots \bigwedge^2 \VV \otimes \OO_{\PW} \to \VV \otimes \OO_{\PW}(2) \to
%   \OO_{\PW}(4) \to 0
% \]
% \end{proof}

\subsection{Geometry of $t_2$ and the double cover of $\PVv$}
\begin{defi}
  Peskine variety of the trivector $t_2\in\bigwedge^3\VV$
  $$\Ypesk := \left\{ [V_1] \in \PVv,\  \rank(t_2(V_1,-,-)) \leq 6\right\} $$
  and the congruence of Definition~\ref{refcongruenceduntriv}
  $$\Btriv := \left\{ [V_2] \in G(2,\VVv)   ,\  t_2(V_2,V_2,-)=0 \right\} $$
  consists in lines of $\PVv$ quadrisecant to $\Ypesk$.
\end{defi}
\begin{prop}[Peskine variety and six secant $\PP_3$]\label{refpeskinetdeux} \
  \begin{itemize}[label=---]
  \item For $S$ general the Peskine variety $\Ypesk$ is smooth irreducible of
    dimension $6$. The preimage of a general point of $\Ypesk$ is a smooth
    rational cubic curve $C\subset \PW$. The intersection $X\cap C$ consist in  $6$
    distinct points defined by the restriction of the quadrics containing $X$ to
    the linear space spanned by $C$.
  \item Any $\PP_3\subset \PW$ intersecting $X$ in a scheme $Z_6$ of length six
    with $h^1(\cI_{Z_6}(2))=0$ is sent to $\PVv$ by $\fX$ to a three
  dimensional projective subspace of $\PVv$ in the kernel of  $t_2$. The map
  $\fX$ contracts the rational cubic curve $C$ containing $Z_6$ to a point
  of the Peskine  variety $\Ypesk\subset \PVv$.
\end{itemize}
\end{prop}
\begin{proof}
  The smoothness and dimension property is an open condition on the trivector
  defining a Peskine variety, but we have to show that this openset is not empty
  for trivectors $t_2$ defined by a general K3 surface of genus $16$. Note that
  the singular locus of $\Ypesk$ might be bigger than the locus where $t_2$ has
  rank at most $4$ and that we can't expect to compute the rank of its jacobian.
  But fortunately, we were able to check the smoothness of the Peskine variety
  $\Ypesk$ on an example using the Benedetti-Song criterion
  (cf. \cite[Lemma~2.8]{benedettisongDV} or \cite[Lemma~4.1]{songspecialDV}) so
  for a general $S$ the Peskine variety $\Ypesk$ is irreducible and smooth of
  dimension $6$.

  Conversely, let $Z_6$ be the intersection of $X$ with a $6$-secant
  $\PP_3$. Remind that $X$ is
  defined by quadrics and that $Z_6=X\cap \PP_3$. So the restriction of $\fX$ to
  this $\PP_3$  is defined by the linear system $|\cI_{Z_6}(2)|$. This is a
  rational map of degree $2$ to from $\PP_3$ to the following three dimensional
  projective space $\PP(\ker(t_2(\fX(C),-,-)))$, where $C$ is the cubic curve
  contracted by $|\cI_{Z_6}(2)|$. Indeed, a general line $l$ bisecant to $C$ is invariant by
  $i_X$ so $\fX(l)\in \Btriv$. But all these lines contain the point $\fX(C)$ so
  they are all in the kernel of $\ker(t_2(\fX(C),-,-)$, hence $\fX(C)\in\Ypesk$.
\end{proof}

% \begin{claim}[Six secant $\PP_3$]
%   A general $\PP_3$ six secant to $X$ is sent to $\PVv$ by $\fX$ to a three
%   dimensional projective subspace of $\PVv$ of Zak's type for $t_2$. The map
%   $\fX$ contracts the rational cubic curve $C$ containing these $6$ points to a point
%   of the Peskine  variety $\Ypesk\subset \PVv$ of $t_2$.

%   Bisecant lines to $C$ are in the congruence $B_{\iX}$ and their image by $\fX$
%   are $4$-secant to $\Ypesk$ through the point $\fX(C)$.
% \end{claim}
% \nbmarge{Ce qui n'est pas clair mais qui a l'air d'être vrai sur des exemples
%   est plutot que cette contruction donne le point general du la varieté de
%   Peskin $\Ypesk$. Autrement dit
%   Un point general de la varieté de Peskine $\Ypesk$ a l'air d'être l'image d'une
%   cubique gauche $6$-secante à $X$ et contractée par le revêtement double}
% \begin{proof}
%   Let $X_6$ be the intersection of $X$ with a $6$-secant $\PP_3$. As $X$ is
%   defined by quadrics, $X_6$ is defined by the restriction of these quadrics to $\PP_3$.
% \end{proof}
\begin{rema}
  Let $l\subset \PW$ be a general $i_X$-invariant line. Then the premimage by
  $\fX$ of the line $\fX(l)$ is
  $$
\fX^{-1}(\fX(l)) = l \cup C_1 \cup C_2 \cup C_3 \cup C_4,
$$
where $(C_i)_{1\leq i \leq 4}$ are disjoint smooth rational cubic curves
six-secant to $X$ and bisecant to $l$.
\end{rema}
\begin{proof}
  The residual of $l$ in $\fX^{-1}(\fX(l))$ is a curve of degree $12$ and genus
  $-3$ contracted by $\fX$ because we saw in the proof of
  Proposition~\ref{refdeftriv2} that the restriction of $\fX$ from $l$ to
  $\fX(l)$ is already a double cover of $\fX(l)$. But we could check on an
  example that this contracted curve is the union of $4$ smooth cubic curves.
\end{proof}
\begin{coro}[Geometric description of invariant lines]
  The order one congruence of invariant lines $B_{\iX}$ consists of lines
  bisecant to  rational cubic curves six secant to $X$.
\end{coro}
\begin{coro}
  The Jacobian divisor $\JacX$ defined by $\fX$ in $\PW$ has an involution $i_2$ preserving
  the intersection of the $\iX$-invariant lines with $\JacX$. The invariant
  locus of $i_2$-contains the divisor $\fX^{-1}(\Ypesk)$ of  $\JacX$.
\end{coro}
\begin{prop}\label{refXdansDVtdeux}
  The trivector $t_2$ is null on the  vector bundle $\NX(-2H)$ included in
  $\VVv\otimes \OO_X$ by the transpose of sequence (\ref{eqvertexbundle}). In
  other words the orthogonal of the vertex bundle $p_S^*(\mG)(-H)$ defined in
  Lemma~\ref{refvertexbundle} gives a map from $X$ to the Debarre-Voisin variety   of $t_2$.
\end{prop}
\begin{proof}
  Let $x$ be a point of $X$. The
  restriction of $s_\gamma$ to $\{x\}$ has rank $4$ and denote by $G_{4,x}$ its
  image in $\VV$ and by $\begin{tikzcd}[cramped]
	\VV & N_{X,x}^\dual
	\arrow["{n_x}", from=1-1, to=1-2]
      \end{tikzcd}$ its quotient map. By definition the coefficents of the composition $n_x \circ
      s_\gamma$ are linear forms vanishing at $x$. So the coefficents of the
      following composition

\[\begin{tikzcd}[cramped]
	{{N_{X,x}}\otimes \OO_{\PW}(-1)} && {V_8^\dual \otimes \OO_{\PW}} & {V_8\otimes\OO_{\PW}} && {{N_{X,x}^\dual}\otimes\OO_{\PW}(1)}
	\arrow["{^t (n_x\circ s_\gamma)}", from=1-1, to=1-3]
	\arrow["\phi", from=1-3, to=1-4]
	\arrow["{n_x\circ s_\gamma}", from=1-4, to=1-6]
\end{tikzcd}\]
are quadratic forms singular at $x$ and vanishing on $X$. So they are in
$G_{4,x}\subset \VV$ hence the trivector $t_2$ is zero on $N_{X,x} = G_{4,x}^\bot$

\end{proof}
\begin{coro}
  From the exact sequence defined in Corollary~\ref{refOmegaS}
  The scroll \[\Delta_S:=Proj(Sym(\Omega^1_S(h)))\] is sent to $\PVv$ but its rulings are
  $4$-secant to the Peskine variety $\Ypesk$. In other words, these lines
  belong to
  the congruence $\Btriv$ and the image of $S$ in $G(2,\VVv)$ defined by
  $\Omega^1_S(h)$ spans only the $34$-dimensional projective space
  $\PP(\ker(\wedge^2 \VVv \xrightarrow{~ t_2~ } \VV))$.
\end{coro}
\begin{proof}
  NB: inclusion in this projective space is obtained from the previous Proposition
  because these lines are the intersection of $2$ Debarre-Voisin spaces.

  Moreover  we have examples where the
  span has indeed projective dimension $34$ so in general we have equality of
  these projective spaces.
\end{proof}
\nbmarge{Question1: Describe the $4$-uple cover of $S$ defined by $\Delta_S
  \cap \Ypesk$; En ces points la cubique gauche contient une regle de $X$ il me semble.}

\bigskip
\nbmarge{Question2: is it related to this composition?
  $$cone(\Delta_S)\subset cone(DV(t_1)) \subset \wedge^6 \VV \xrightarrow{~ t_2~
  } \VVv$$
}

%%%%%%%%%%%%%%%%%%%%%%%%%%%%%%%%%
\section{Further results and observations}
\subsection{The trivector $t_1$ and $\Hilb^2(S)$}
From \cite{oberdieckGW} the divisor $D_{30}$ associated to Hilbert squares of K3
surfaces of genus $16$ is not HLS. Moreover, the case of singular Debarre-Voisin
varieties was already done in \cite{DV}. Hence for a general
element of this divisor, there is a trivector $t_1\in \wedge^3\VVv$ such that
$DV(t_1))=\Hilb^2(S)$. In \cite{DHOV} we give rank $4$ vector bundle $Q_4$ on
$\Hilb^2(S)$ that is modular in the meaning of \cite[modular]{ogmodularsheaves}
and a geometric interpretation of the associated
map from $\Hilb^2(S)$ to $G(6,10)$. To show that this map has the same
polarization as $DV(t_1)$ we need to know that $\Qmodular$ is generated in codimension
at least $2$. From Th.~\ref{refequationsX} and Prop.~\ref{refXsmooth} we have
the following
\begin{lemm}
  Let $\Tauto{F}$ (resp. $\Tauto{S_2(F)}$ ) be the tautological transform of $F$
  (resp. of the symmetric square of $F$). Consider
  the rank $4$ vector bundle $\Qmodular$ defined in \cite[Lemma~8.1]{DHOV} by the following exact
  sequence
  \[
     0 \to \Qmodular \to S_2(\Tauto{F}) \xrightarrow{\  ev^+\  } \Tauto{S_2(F)} \to 0.
   \]
   Then $\Qmodular$ is globally generated.
\end{lemm}
\begin{proof}
  Any subscheme $z$ of length $2$ of $S$  defines a locally Cohen-Macaulay
  scheme $Z=\PP(F_{z})$ of dimension $1$, degree $2$ and genus $-1$. So $Z$ is a
  divisor of degree $(2,0)$ on a smooth quadric in a $3$-dimensional projective
  space $\PP_{3,Z}\subset \PW$. The fiber  $\Qmodular_{\{z\}}$ of $\Qmodular$ over
  $z$ is just the space of quadrics of $\PP_{3,Z}$ containing $Z$. Remind from
  \cite[Lemma~8.3]{DHOV} that $H^0(\cI_X(2))=H^0(\Qmodular)$, so we have to show that the
  the quadrics containing $X$ restricted  to $\PP_{3,Z}$ generate
  $\Qmodular_{\{z\}}=H^0(\cI_{Z|\PP_{3,Z}}(2))$. Now remark that any vector
  subspace of $H^0(\cI_{Z|\PP_{3,Z}}(2))$ dimension at most $3$ defines a locus
  in $\PP_{3,Z}$ containing a line $l$ bisecant to $Z$. From
  Prop.~\ref{refXsmooth}, the line $l$ can't be  a ruling of $X$.

  Remark that for any linear subspace $\PP_7$ of $\PW$ of
  codimension $2$ the projection $p_S(X\cap \PP_7)$ is an element of $|\wedge^2
  F|$, so it is a generator of $Pic(S)$. So $X\cap \PP_7$ is always an
  irreducible curve, and the line $l$ can't be included in $X\cap \PP_{3,Z}$.

  Now recall from Th.~\ref{refequationsX} that the ideal of $X$ is generated by
  the quadrics of $\VV$, so their restrictions defines the ideal of $X\cap
  \PP_{3,Z}$ and the image of $\VV$ to $\Qmodular_{\{z\}}$ can't have dimension
  $3$ or less. So $\Qmodular$ is globally generated.
\end{proof}

Unfortunately we are not able to understand $t_1$ explicitely, but we used the
following algorithm to compute it from $X$.
\begin{algo}[Computing $t_1$]
  \
  \begin{enumerate}[i)]
  \item Find $5$ points $(p_i)$ on $X$ in general position.
  \item Compute the $5$ rulings $(l_i)$ of $X$ containing these points via the
    intersection of $X$ with its tangent space at $p_i$.
  \item Compute the $10$ vector subspaces $K_{i,j}\subset \VV$ dimension $6$
    defined by the element of $\VV$ vanishing on the linear space spaned by
    $l_i\cup l_j$.
  \item Compute the $45$ intersections $\delta_{i,j,i',j'}=K_{i,j}\cap
    K_{i',j'}$ and compute the vector subspace $V_{35}$ of $\wedge^2\VV$ spanned
    by the $\delta_{i,j,i',j'}$. And denote by $N_{10}$ the kernel
    \[
      0 \to N_{10} \to \wedge^2\VVv \to V_{35}^\dual \to 0,
    \]
    then check that $\dim(N_{10})=10$.
  \item Compute the kernel of the composition
    \[
      N_{10}\otimes \VV \to \VVv \otimes \VVv \to S_{2,1}\VVv,\]
    and check that it has dimension $1$. Then verify that it gives an
    isomorphism $N_{10} \simeq \VVv$.
  \item Then the composition $\VVv \to N_{10} \to \wedge^2\VVv$ explicits an
    element of $\wedge^3 \VVv$ and it is the unique trivector $t$ such that the $45$ lines
    $\delta_{i,j,i',j'}$ belong to $Y^2_t$.
  \end{enumerate}
\end{algo}
\begin{rema}\
  \begin{itemize}[---]
  \item Any trivector $t$ vanishing on the $10$ subspaces $(K_{i,j})$ in the
    previous algorithm is such that $\delta_{i,j,i',j'}$ is in the congruence of
    lines $Y^2_t \subset G(2,\VV)$. Hence such a trivector equals  the
    output of the  previous algorithm.
  \item To know if this algorithm is successfull for a generic $S$ to compute $t_1$, we
    need to show the slope stability of $\Qmodular$ to apply the uniqueness
    result of \cite{ogmodularsheaves}.
  \end{itemize}
\end{rema}
By analogy with \cite[Prop~8.4]{ogmodularsheaves} treating the case of
Debarre-Voisin varieties with cyclic Picard group, we have the following
\begin{lemm}[Slope stability]
  Let $Q$ be a globally generated vector bundle of rank $4$ on $\Hilb^2(S)$ on a
  K3 surface $(S,h)$ of genus $16$ with $Pic(S)=\ZZ.h$. Assume that $Q$ has the same Chern classes as
  $\Qmodular$ (i.e $Q$ is \cite[modular]{ogmodularsheaves}). Then $Q$ is
  $c_1(Q)$-slope stable.
\end{lemm}
\begin{proof}
  Let $\cB$ be a torsion-free desemistabilizing quotient of $Q$ with $0<rank(\cB)<4$. Decompose
  $c_1(\cB)$ in $NS(\Hilb^2(S))=L\ZZ\oplus \delta \ZZ$ as   $\alpha L -\beta
  \delta$, where $L$ is the line bundle induced on $\Hilb^2(S)$ by the
  polarization $(S,h)$ and $\delta$ is such that $2\delta$ is the class of the
  exceptionnal divisor of the Hilbert-Chow morphism.  From \cite[\S 3.3]{DHOV}
  we have $c_1(Q)=2L-7\delta$ and the slope of the movable cone is $15/4$. We
  have
  \begin{equation*}
    c_1(\cB)\cdot c_1(Q)^3 \leq \frac{3}{4} c_1(Q)^4,\ \mbox{ and\quad } L\cdot c_1(Q)^3=3960,\
    \delta\cdot c_1(Q)^3=924.
  \end{equation*}
 By assumption $Q$ is globally generated. Hence  $\cB$ is also generated and $c_1(\cB)$ is a
 movable class. So we have
 \begin{equation*}
   \label{eq:destabiB}
   \begin{cases}
     \alpha \cdot 3960 - \beta\cdot 924 \leq 1089 \\
     0<\alpha,\quad \frac{\beta}{\alpha} \leq \frac{15}{4}
   \end{cases}
 \end{equation*}
 This implies $0<\alpha \leq \frac{11}{2}$ hence $\alpha \in \{1,2\}$. But the previous
 inequality gives in these two cases:
 \[
 \begin{cases}
   \alpha = 1 \Rightarrow \beta \geq \frac{3960-1089}{924}\geq 3.1 \Rightarrow
   \beta \geq 4 \Rightarrow \frac{\beta}{\alpha} \geq 4
   \\
   \alpha = 2 \Rightarrow \beta \geq \frac{2\cdot 3960-1089}{924}\geq 7.3 \Rightarrow
   \beta \geq 8 \Rightarrow \frac{\beta}{\alpha} \geq 4,
 \end{cases}
\]
So they both contradict the movable condition $\frac{\beta}{\alpha} \leq
\frac{15}{4}$. In conlusion $Q$ is slope stable.
% cf file:
% DV-schubert.m2
% 1/2*(chern(1,Q4))^3*E -- 924
% 1/2*(chern(1,Q4))^4 -- 1452
% 1/2*(h1+h2)*(chern(1,Q4))^3 -- 3960
%-- output --
% i36 : Q4= symmetricPower(2,Taut(F)) - Taut(symmetricPower(2,F))
% o36 = Q4
% o36 : an abstract sheaf of rank 4 on Bl
% i37 : chern Q4
% o37 = 1 + (- 7E + 2h1 + 2h2) + (- 21E*h2 + 3h1*h2 - 18delta + 48p1 + 48p2) + (20E*delta - 36h2*delta + 48h2*p1 - 648E*p2 + 48h1*p2) + 18p1*p2
% o37 : RBl
% i38 : 1/2*(chern(1,Q4))^3*E -- 924
% o38 = 924p1*p2
% o38 : RBl

% i39 : 1/2*(chern(1,Q4))^4 -- 1452
% o39 = 1452p1*p2
% o39 : RBl

% i40 : 1/2*(h1+h2)*(chern(1,Q4))^3 -- 3960
% o40 = 3960p1*p2
% o40 : RBl

\end{proof}
\subsection{Orthogonality of trivectors}
\begin{defi}[Composition of trivectors]
 For any $(a,b)$ in $\bigwedge^3 \VVv\times \bigwedge^3 \VV $ let us define
 the composition of $a$ by $b$ as the following endomorphism:
 $$
  b\circ a :~ \VV \xrightarrow{\ \ a\ \ } \bigwedge^2 \VVv \xrightarrow{\ \ b\ \ } \VV
  $$
  and the following orthogonality condition:
  $$
\bot_{b}:=\left\{ a \in \bigwedge^3\VVv,~ b\circ a =0 \right\}.
$$
Note that the condition $a\in \bot_{b}$ is equivalent to the vanishing of the trace
of
$$\VV\otimes \OO_{\PVv\times \PV} \xrightarrow{\ \ B\circ A\ \ } \VV\otimes \OO_{\PVv\times \PV}(1,1)$$
where $A$ and $B$ are the maps
$$\VV\otimes \OO_{\PVv} \xrightarrow{\ \ A\ \ } \VVv\otimes \OO_{\PVv}(1),~ \VVv\otimes \OO_{\PV} \xrightarrow{\ \ B\ \ } \VV\otimes \OO_{\PV}(1)$$
canonically defined by $a$ and $b$.

\end{defi}
\begin{prop}\label{reforthogonalitetriv}
  There is an example of K3-surface of genus $16$ such that the following linear space
  $$
\bot_{t_2}:=\left\{ t \in \bigwedge^3\VVv,~ t_2\circ t =0 \right\}
$$
has the expected codimension $100$. Hence this open condition is true for a
general K3-surface of genus $16$.

  With the previous algorithm we were able to compute $t_1$ for this surface  and
  observe  that  $$t_1\in \bot_{t_2}.$$

  Moreover for these values of $(t_1,t_2)$,  we were able to prove that the intersection
  $$
\left\{ t_{1,h} \in \wedge^3\VVv,~h \in End(\VV) \right\} \cap
\bot_{t_2}
$$
contains only the multiples of $t_1$, where $t_{1,h}$ is defined for any
endomorphism $h$ of $\VV$ by
  $$
t_{1,h}:=(v_1,v_2,v_3)\mapsto t_1(h(v_1),v_2,v_3)+t_1(v_1,h(v_2),v_3)+t_1(v_1,v_2,h(v_3)).
$$

In other words, $t_1$ is an isolated point in the intersection of its
$PGL_{10}$-orbit in $\PP(\wedge^3\VVv)$ with $\PP(\bot_{t_2})$.
\end{prop}

\subsection{Discriminant and Ramification}
\begin{defi}
  Let $\mgamma: W_{10}^\dual \otimes  \OO_\PV(-1)\to W_{10} \otimes  \OO_\PV$ be
  the symmetric map defined by $\gamma$. Denote by $\discV$ the discriminant
  hypersurface of $\PVv$ defined by $\det(\mgamma)$.

  Let $\JacX$ be the hypersurface of $\PW$ defined by the jacobian of $\fX$, and
  by $\ramiVv$ the ramification hypersurface of $\PVv$ dual to $\discV$.
\end{defi}
\subsubsection*{--- The discriminant divisor $\discV $ ---}
\begin{prop}[Discriminant]
  For a generic K3-surface $S$ of genus $16$,
  \begin{enumerate}[i)]\label{refsingdiscri}
  \item The discriminant $\discV$ is a normal hypersurface of $\PV$.
  \item The determinantal locus $Fit^1(\mgamma)$ defined by $\rank(\mgamma) \leq 8$ has the
    expected codimension. So $\dim(Fit^1(\mgamma))=6$,
    $\deg(Fit^1(\mgamma))=165$.
  \item The $6$-dimensional part $Sing_{top}(\discV)$ of the singular locus of $\discV$ has
    degree $225$, and the residual component $X_{60}$ of $Fit^1(\mgamma)$ in $Sing_{top}(\discV)$
    is the image of $\Proj(p_S^*(\mGdual)(H))$ in $\PV$ and $\mgamma$ has
    generic rank $9$ on $X_{60}$.
  \end{enumerate}
\end{prop}
Moreover from computations in \cite[Macaulay2]{macaulay2} we suggest the following
\begin{conj}
  The Peskine variety $\Ypeskdual$ is included in $Fit^1(\mgamma)$ and there is
  a $6$ dimensional locus $X_{150}$ such that
  $$ Sing_{top}(\discV) = X_{150} \cup  \Ypeskdual \cup X_{60},\quad
  \deg(X_{150})=150.$$
\end{conj}
\begin{proof}[(Proof of Proposition~\ref{refsingdiscri})]
  The first two items are open properties, so we  checked on an example by
  restriction to some $\PP_3\subset \PV$. On this example we found that the
  singular locus of the restriction of $\discV$ to some $\PP_3$ was zero
  dimensional of degree $225$. This proves $i)$ and $ii)$ and shows that
  $\deg(Sing_{top}(\discV))\leq 225$.

  So to obtain $iii)$ we have to show that the generic element of
  $\Proj(p_S^*(\mGdual)(H))$ is sent to a singular point of $\discV$ of rank
  $9$.
  Let $q$ be a quadric of rank $9$. The tangent space to the universal
  discriminant in $\PP(S_2(W_{10}))$ at $q$ is given by the quadrics $q'$
  containing the singular point of $q$. So $q$ is in $Sing(\discV)$ if and only
  if its singular point is in the base locus of the linear system of quadrics
  $\VV \subset S_2(W_{10})$. In other words
  $$q\in Sing(\discV)\setminus Fit^1(\mgamma) \iff sing(q) \mbox{ is a point of
  } X.$$

  From Lemma~\ref{refvertexbundle}, the relative projective space
  $\Proj(p_S^*(\mGdual)(H))$ is a six dimensional variety such that its image
  $X_{60}$ in $\PV$ satisfy the previous condition. An example is enough to show
  that  there exists quadrics of rank $9$ with singular locus a point on $X$. So
  $\Proj(p_S^*(\mGdual)(H))$ and $X_{60}$ are birational. So the degree of
  the image $X_{60}$ is given by the third segre class of $p_S^*(\mGdual)(H))$
  and is $60$.
% cf. file:
% -- genre16-degreinvolution.m2
% i73 : chern(dual((pS^*(G4))(-1)))

% o73 = 1 + (4H  - h) + (3h*H  - 39p) + 24p*H
%              1             1               1

%                RS[a , H ]
%                    1   1
% o73 : ----------------------------
%       (- a  - H  + h, - a H  + 9p)
%           1    1         1 1

% i74 : segre(3,dual((pS^*(G4))(-1)))

% o74 = 60p*H
%            1

%                RS[a , H ]
%                    1   1
% o74 : ----------------------------
%       (- a  - H  + h, - a H  + 9p)
%           1    1         1 1
\end{proof}
\subsubsection*{--- The Ramification divisor $\ramiVv$ ---}
\begin{prop}[Kummer quartics]\label{refkummer}
  Denote by $\EVv$ the kernel of
  the restriction of the trivector map $t_2:\VVv \otimes \OO_{\PVv} \to \VV
  \otimes \OO_{\PVv}(1)$ to the Peskine variety $\Ypesk$.

  We have for a generic $K3$-surface $S$ of genus $16$,
  \begin{enumerate}[i)]
  \item  The sheaf $\EVv$ is locally free of rank
    $4$ on $\Ypesk$. (cf. Prop~\ref{refpeskinetdeux})
    \[ 0 \to \EVv \to \VVv \otimes \OO_{\Ypesk} \xrightarrow{\quad t_2 \quad} \VV
    \otimes \OO_{\Ypesk}(1) \to (\EVv)^\dual(1) \to 0.
  \]
  \item For a general point $p$ of the six dimensional variety $\Ypesk$ the
    intersection of $\Ypesk$ with the kernel space  $\PP(\EVv_{\{p\}})$ is the
    disjoint union of the point $p$ and a smooth cubic surface $S_{3,p}$.
  \item For a general point $p$ of $\Ypesk$ the restriction of the ramification
    $\ramiVv$ to $\PP(\EVv_{\{p\}})$ is the union of a Kummer quartic surface
    $\KuVv_{\{p\}}$ and the double structure on $S_{3,p}$. Moreover this point $p$ is the
    only one of the $16$ nodes of $\KuVv_{\{p\}}$ that is also on $\Ypesk$.
  \item Any point $r$ of $\ramiVv \setminus \Ypesk$ is by definition on a unique
   line $l_r$  quadrisecant to $\Ypesk$. Now assume that  $r$ is general in
   $\ramiVv\setminus \Ypesk$ then
   the intersection
   $l_r\cap \Ypesk$ consist in $4$ distinct points $p_1,\dots,p_4$ and the point
   $r$ is on the four  Kummer quartics $\KuVv_{\{p_i\}}$ inducing the following
   decomposition of the tangent spaces
   \[
    T_{r} (\ramiVv) = \bigoplus_{i=1}^4 T_r(\KuVv_{\{p_i\}}),\quad \frac{\VVv}{L_r} = \bigoplus_{i=1}^4 \frac{\EVv_{\{p_i\}}}{L_r}.
   \]
  \end{enumerate}
\end{prop}
\begin{proof}
  We already saw in Proposition~\ref{refpeskinetdeux} that $\Ypesk$ is smooth for a
  general $K3$ surface of genus $16$. This is stronger than saying that locus
  where $t_2$ has rank at most $4$ is empty, so $i)$ was already done.

  For $ii)$ recall that any line of $\PP(\EVv_{p})$ containing $p$ is a
  quadrisecant line to $\Ypesk$. So there is always a cubic surface in
  $\PP(\EVv_{p}) \cap \Ypesk$ and its an open property to say that it is smooth
  and disjoint from $p$. We have just checked on an example that it is not
  empty.

 For $iii)$, remind from Proposition~\ref{refpeskinetdeux} that for a general point $p$
 of $\Ypesk$ the premimage $\fX^{-1}(p)$ is a cubic curve six-secant to $X$. So
 $\Ypesk$ is in the singular locus of the ramification $\ramiVv$ of $\fX$. Hence
 the intersection $\ramiVv\cap \PP(\EVv_p)$ contains the double structure on the
 cubic surface $S_{3,p}$, and its residual is a quartic surface. To understand
 this surface,  consider
 the restriction of $\fX$ to the $3$-dimensional projective space $\PP_p$ spaned by the
 cubic curve $\fX^{-1}(p)$. It is the linear system of quadrics  of $\PP_p$
 containing the six points of $X\cap \PP_p$. It gives a double cover of
 $\PP(\EVv_p)$ with Jacobian a Weddle surface in $\PP_p$ and ramification a
 Kummer quartic $\KuVv_{\{p\}}$ with nodes  at the image by $\fX$ of the $6$-secant
 cubic and the $15$ lines bisecant to the $6$ points of $X\cap \PP_p$
 (cf. \cite[Chap XV, \S 98]{Hudson}).

 Now that $iii)$ is proved, $iv)$ is just an open condition satisfied  by an
 example computed with \cite[Macaulay2]{macaulay2}.
\end{proof}

\begin{prop}
  In \ref{refsgamma} we introduced a map $s_{\gamma}$ defined by the linear
  Syzygies $\gamma$ of $X$. Now consider the map $s'_\gamma$ defined by $\gamma$
  over $\PVv$
   \[
    V_8 \otimes \OO_{\PVv} \xrightarrow{\ s'_\gamma \ } W_{10}\otimes \OO_{\PVv}(1).
  \]
  For a generic $K3$-surface of genus $16$,
  \begin{enumerate}[i)]
  \item the ideal $Fit^0(s'_\gamma)$ of maximal minors of $s'_\gamma$
  defines a locus of the expected codimension $3$ (hence degree $120$) and the locus where
  $s'_\gamma$ has rank $7$ is not empty.
  \item The restriction of $s'_\gamma$ to the Peskine variety $\Ypesk$ has
    generic rank $6$. So it is an irreducible component of  $Fit^0(s'_\gamma)$
    and denote by $R_{60}$ the union of the other components of
    $Fit^0(s'_\gamma)$.
  \item The locus $R_{60}$ is reduced of dimension $6$ and degree $60$ and contains the image by
    $\fX$ of the smooth conics $4$-secant to $X$.
  \end{enumerate}
\end{prop}
\begin{proof}
  We checked $i)$ on an example so this open property is not empty. To obtain
  $ii)$ recall from $(\ref{eq:noyaucongruence})$ in the proof of
  Prop~\ref{refdeftriv2} that for $p$ general in $\PW$, we have $\fX(p)$ is in
  the kernel of ${}^t s_{\gamma,\{p\}}$. So we have ${}^t \sigma_\gamma(p\otimes
  \fX(p)) = 0$ with
  \[
     W_{10}^{\dual} \otimes \VVv \xrightarrow{\quad{}^t\sigma_\gamma\quad} V_8^\dual.
   \]
   So for $q\in \PVv$ and any $p$ in $\fX^{-1}(q)$ we have ${}^t\sigma_\gamma(p\otimes
   q) = 0$, and the vector space spaned by $\fX^{-1}(q)$ is in the kernel of
   ${}^t s'_{\gamma}$ over the point $q$. For general $q$ in $\Ypesk$, $\fX^{-1}(q)$ is
   a smooth rational cubic curve so the rank of $s'_{\gamma}$ is at most $6$ at
   $q$, and we have examples where it is exactly $6$. Moreover $\Ypesk$ is
   irreducible of degree $15$ of codimension $3$ so it is a component of degree
   $60$ of  $Fit^0(s'_\gamma)$, so $R_{60}$ has degree $120-60$.
% ----------------------------------------------------
% ----verif rapide que R60 est reduit  ---------------
%
% load("initTestmod101.m2")
% use R9
% P3=ideal(v_5+35*v_6-30*v_7+27*v_8-22*v_9,v_4+49*v_6+46*v_7-17*v_8+7*v_9,v_3+38*v_6-39*v_7-17*v_8+11*v_9,v_2+39*v_6-39*v_7+28*v_8+10*v_9,v_1-19*v_6-7*v_7-43*v_8+44*v_9,v_0+48*v_6-28*v_7-22*v_8+43*v_9)

% tmp=substitute(tmm810,R9/P3)
% Fit0=minors(8,tmp);
% IP3=saturate(P3+Ipesk)
% IP3b=substitute(IP3,R9/P3)
% R60=Fit0:(IP3b^2);
% dim (R60+IP3b)
% degree R60
% rR60=radical R60;
% degree rR60 --OK c'est reduit.

% -- i44 : degree R60

% -- o44 = 60

% -- i45 : rR60=radical R60;

% --                                                                                                              R9
% -- o45 : Ideal of ----------------------------------------------------------------------------------------------------------------------------------------------------------------------------------------------
% --                (v  + 35v  - 30v  + 27v  - 22v , v  + 49v  + 46v  - 17v  + 7v , v  + 38v  - 39v  - 17v  + 11v , v  + 39v  - 39v  + 28v  + 10v , v  - 19v  - 7v  - 43v  + 44v , v  + 48v  - 28v  - 22v  + 43v )
% --                  5      6      7      8      9   4      6      7      8     9   3      6      7      8      9   2      6      7      8      9   1      6     7      8      9   0      6      7      8      9

% -- i46 : degree rR60

% -- o46 = 60
In a similar way, if $\fX^{-1}(q)$ is a conic $4$-secant to $X$, then it span a
kernel of rank $3$  of ${}^t s'_{\gamma}$ so $q$ is in $R_{60}$.
\end{proof}
\subsection{Conjectures}
Unfortunately we were not able to achieve the elimination computation to obtain
$\ramiVv$, and we were not able to interpret the Kummer quartics that we found
after taking points on $\Ypeskdual$ so we can conjecture that
\begin{conj}\
  \begin{enumerate}[i)]
  \item The singular locus of $\ramiVv$  is the union of three irreducible varieties of
  dimension $6$
  $$Sing_{top}(\ramiVv) = Sec(X) \cup R_{60} \cup \Ypesk$$
  corresponding to the image by $\fX$ of line bisecant to $X$, conic $4$-secant,
  cubic $6$-secant.
  \item (Kummer quartics for $t_1$) For a general point $p$ of $\Ypeskdual$ the restriction of the
    discriminant $\discV$ to the three dimensional projective space defined by
    the kernel of $t_1$ over $p$ is also twice a cubic surface and a Kummer
    quartic nodal at $p$.
  \item By analogy with the case of Kummer quartics beeing projectively
    isomorphic\footnote{via the choice of one of the $10$ invariant quadrics}
    to its dual, we can still ask if there is an linear isomorphism between $\VV$ and
    $\VVv$ sending $\discV$ to $\ramiVv$ and $t_1$ to $t_2$. Note that recently Meng gives
    evidences in \cite{meng} that we should not expect this for $S$ generic.
  % \item Note that if $iii)$ is true then $DV(t_2)$ is also isomorphism to
  %   $\Hilb^2(S)$. But in Prop~\ref{refXdansDVtdeux} we have a natural map from $X$ to
  %   $DV(t_2)$ and $X$ can't be the diagonal of $\Hilb^2(S)$. So we conjecture
  %   that the vanishing locus $X'$ of the determinant of the following map  is isomorphic to  $X$.
  %   \[ V_8 \otimes \OO_{\Hilb^2(S)} \longrightarrow  \mG^{[2]},
  %   \] where $\mG^{[2]}$ is the tautological  bundle on $\Hilb^2(S)$ constructed
  %   from $\mG$.
  \item In Prop~\ref{refXdansDVtdeux} we have a natural map from $X$ to
    $DV(t_2)$. So we should undersand the relationship between $X$ and
    the vanishing locus $X'$ of the determinant of the following map
    \[ V_8 \otimes \OO_{\Hilb^2(S)} \longrightarrow  \mG^{[2]},
    \] where $\mG^{[2]}$ is the tautological  bundle on $\Hilb^2(S)$ constructed
    from $\mG$.
  \item Note that the divisor introduced in $iv)$ plays a role in the following
    observation that we conjecture to be still satisfied by a general point $p$ of $R_{60}$. Denote by $v\in V_8$
    a generator of the kernel of $s'_\gamma$ over $p$. Then the compotion
    of $v$ with $s_\gamma$ gives a section of $W_{10} \otimes \OO_{\PW}(1)$
    vanishing on a $3$-dimensional projective space $\PP_{3,v}$. Moreover the
    intersection $\PP_{3,v}\cap X$ consist in two rulings of $X$. Denote by
    $s_1,s_2 \in S$ these two points. The vector $v$ is in the intersection of
    the two subspaces $\mGdual_{\{s_1\}} \cap \mGdual_{\{s_2\}}$ in $V_8$, so
    $\{s_1,s_2\}$ is in the divisor $X'$ defined in $iv)$. Moreover $\fX(\PP_{3,v})$ is a
    quadric surface and the restrition of $s'_\gamma$ to the $3$ dimensional projective space
    spaned by this quadric has a constant kernel spanned by $v$.
    \item Note also that if $iii)$ is true then the
      Proposition~\ref{reforthogonalitetriv} explicits a divisorial condition on
      $t_1$ by considering the pairs $(t_1,\psi)$ of $\PP(\wedge^3 \VVv) \times
      \PP(Hom(\VVv,\VV))^{\simeq}$ such that $t_1 \in \bot_{\wedge^3 \psi(t_1)}$.
  \end{enumerate}
\end{conj}
\section{An example defined over $\FF_{8}$}\label{sectionchar2}
The explicit construction in Section~2 can be done on any base field but to know
what kind of open properties are still valid on a specific finite field we need
to try. In particular this construction still produces examples in
characteristic $2$. Also we were not able to find smooth examples defined over
$\FF_2$ or $\FF_4$ we found the following one defined over $\FF_8$ by taking
random choices for $M$ and $N$ in Section~2.

%%%%%%%%
%% cf X-exempleF8lisse.m2
%%%%%%%
\begin{exam}
  Let $\FF_8=\FF_2[a]/(a^3-a-1)$. For the choice of $M$ consider the following
  quadrics of $\PP_3$
  {\small \begin{align*}
   m_0={}&x_{0}^{2}+x_{0}x_{1}+\left(a^{2}+a\right)x_{1}^{2}+\left(a^{2}+a\right)x_{0}x_{2}+x_{1}x_{2}+x_{0}x_{3}+a^{2}x_{1}x_{3}+a\,x_{2}x_{3}+\left(a^{2}+a+1\right)x_{3}^{2}\\
    m_1={}&
            \left(a^{2}+1\right)x_{0}x_{1}+x_{1}^{2}+\left(a+1\right)x_{0}x_{2}+a^{2}x_{1}x_{2}+a\,x_{2}^{2}+x_{0}x_{3}+a^{2}x_{2}x_{3}+a\,x_{3}^{2} ,
  \end{align*}}
and consider for $N$ the following choice\\
  \resizebox{1.1\linewidth}{!}{$\left(\!\begin{array}{cccccccccccccccccccc}
       a&a^{2}&1&a^{2}+a&1&a^{2}+a&a+1&a^{2}&a&a^{2}+1&a^{2}+a+1&a^{2}&a^{2}+a+1&a^{2}+1&a^{2}&a+1&a^{2}+1&a+1&a^{2}&1\\
       a+1&a&a^{2}+a&a+1&a&a^{2}&a^{2}+a&0&a&a+1&1&a^{2}&a+1&1&0&0&a^{2}+a+1&0&a&a+1
       \end{array}\!\right)$}
in the following basis\footnote{This is the base we have got after using the
  PieriMaps package. To avoid mistakes we don't change it.} for $S_{2,1}V^\dual$\\
{\footnotesize
$(-x_{1}\otimes x_{0}^{2}+x_{0}\otimes x_{0}x_{1}$, $-x_{2}\otimes x_{0}^{2}+x_{0}\otimes x_{0}x_{2}$, $-x_{2}\otimes x_{0}x_{1}+x_{1}\otimes x_{0}x_{2}$, $-x_{3}\otimes x_{0}^{2}+x_{0}\otimes x_{0}x_{3}$, $-x_{3}\otimes x_{0}x_{1}+x_{1}\otimes x_{0}x_{3}$, $-x_{3}\otimes x_{0}x_{2}+x_{2}\otimes x_{0}x_{3}$, $-x_{1}\otimes x_{0}x_{1}+x_{0}\otimes x_{1}^{2}$, $-2\,x_{2}\otimes x_{0}x_{1}+x_{1}\otimes x_{0}x_{2}+x_{0}\otimes x_{1}x_{2}$, $-x_{2}\otimes x_{1}^{2}+x_{1}\otimes x_{1}x_{2}$, $-2\,x_{3}\otimes x_{0}x_{1}+x_{1}\otimes x_{0}x_{3}+x_{0}\otimes x_{1}x_{3}$, $-x_{3}\otimes x_{1}^{2}+x_{1}\otimes x_{1}x_{3}$, $-x_{3}\otimes x_{1}x_{2}+x_{2}\otimes x_{1}x_{3}$, $-x_{2}\otimes x_{0}x_{2}+x_{0}\otimes x_{2}^{2}$, $-x_{2}\otimes x_{1}x_{2}+x_{1}\otimes x_{2}^{2}$, $-2\,x_{3}\otimes x_{0}x_{2}+x_{2}\otimes x_{0}x_{3}+x_{0}\otimes x_{2}x_{3}$, $-2\,x_{3}\otimes x_{1}x_{2}+x_{2}\otimes x_{1}x_{3}+x_{1}\otimes x_{2}x_{3}$, $-x_{3}\otimes x_{2}^{2}+x_{2}\otimes x_{2}x_{3}$, $-x_{3}\otimes x_{0}x_{3}+x_{0}\otimes x_{3}^{2}$, $-x_{3}\otimes x_{1}x_{3}+x_{1}\otimes x_{3}^{2}$, $-x_{3}\otimes x_{2}x_{3}+x_{2}\otimes x_{3}^{2})$.
}
%%
% A=QQ[x_0..x_3]
% pp4=matrix {{A_0^2, A_0*A_1, A_0*A_2, A_0*A_3, A_1^2, A_1*A_2, A_1*A_3, A_2^2, A_2*A_3, A_3^2}};
% pp3=matrix {{-A_1, -A_2, 0, -A_3, 0, 0, 0, 0, 0, 0, 0, 0, 0, 0, 0, 0, 0, 0, 0, 0}, {A_0, 0, -A_2, 0, -A_3, 0, -A_1, -2*A_2, 0, -2*A_3, 0, 0, 0, 0, 0, 0, 0, 0, 0, 0}, {0, A_0, A_1, 0, 0, -A_3, 0, A_1, 0, 0, 0, 0, -A_2, 0, -2*A_3, 0,
%        0, 0, 0, 0}, {0, 0, 0, A_0, A_1, A_2, 0, 0, 0, A_1, 0, 0, 0, 0, A_2, 0, 0, -A_3, 0, 0}, {0, 0, 0, 0, 0, 0, A_0, 0, -A_2, 0, -A_3, 0, 0, 0, 0, 0, 0, 0, 0, 0}, {0, 0, 0, 0, 0, 0, 0, A_0, A_1, 0, 0, -A_3, 0, -A_2, 0, -2*A_3, 0, 0,
%        0, 0}, {0, 0, 0, 0, 0, 0, 0, 0, 0, A_0, A_1, A_2, 0, 0, 0, A_2, 0, 0, -A_3, 0}, {0, 0, 0, 0, 0, 0, 0, 0, 0, 0, 0, 0, A_0, A_1, 0, 0, -A_3, 0, 0, 0}, {0, 0, 0, 0, 0, 0, 0, 0, 0, 0, 0, 0, 0, 0, A_0, A_1, A_2, 0, 0, -A_3}, {0, 0,
%        0, 0, 0, 0, 0, 0, 0, 0, 0, 0, 0, 0, 0, 0, 0, A_0, A_1, A_2}};

% pp4
% pp3
% A
% B=A[y_0..y_3]
% s2V=gens (ideal(gens B))^2
% tmp=(s2V)*pp3

\medskip
  Define $\beta$ by the following
  matrix and consider the four  choices of  $p_{\pi_1}=
  \left(\begin{smallmatrix}
    0&1&0&0\\
    0&0&1&0\\
    0&0&0&1
\end{smallmatrix}\right)$, $p_{\pi_2}=
\left(\
  \begin{smallmatrix}
    1&0&0&0\\
    0&0&1&0\\
    0&0&0&1
  \end{smallmatrix}\right)$, $p_{\pi_3}=
\left(
  \begin{smallmatrix}
    0&1&0&0\\
    1&0&1&0\\
    0&0&0&1
  \end{smallmatrix}
\right)$, $p_{\pi_4}=
\left(
  \begin{smallmatrix}
    1&0&0&0\\
    0&1&0&0\\
    0&0&1&0
  \end{smallmatrix}
\right)$. The saturations of the ideals defined by the maximal minors of
$(p_{\pi_i} \circ \beta)_{1\leq i \leq 4}$
give $4+2+2+2$ linearly independent quadrics of $\PW$. The scheme defined by
these ten quadrics $(q_i)_{0\leq i\leq 9}$ have
\begin{itemize}
\item[---] the ``expected'' Hilbert Polynomial
  $\frac{7}{2}X^{3}+3\,X^{2}+\frac{3}{2}X+2$ and the same Betti table as in
  \ref{refbettitable}.
\item[---] Theorem\ref{refequationsS} is still valid  on this example. (Ideal of $S$
  generated by the Plücker quadrics)
\item[---]  The kernel of $s_\gamma \otimes s_\gamma$ is a rank one constant
  subsheaf of $V_8\otimes V_8 \otimes \OO_{\PW}$ that gives a skew-symmetric
  isomorphism $V_8 \simeq V_8^\dual$, so $\phi$ defined in \ref{refphiV8} is still
  unique and non degenerate.
\item[---] $X$ is smooth. (Checked with the package FastMinors with  $30$ good
  minors of size $6$ of the map $J_\gamma$, and compute the  saturation ($\simeq 2h$)).

% -------------
% -- verification de la lissité de X. (long)

% loadPackage("FastMinors")
% use R
% tutu2=chooseGoodMinors(30,6,MVW);

% tmp2=IX+tutu2;

% -- i52 : tutu2=chooseGoodMinors(30,6,MVW);

% -- o52 : Ideal of R

% -- i53 : tmp2=IX+tutu2;

% -- o53 : Ideal of R

% time tmp2s=saturate tmp2;
% dim tmp2s
% tmp2s
% -- i180 : time tmp2s=saturate tmp2;
% --  -- used 12505.3s (cpu); 6395.03s (thread); 0s (gc)

% -- o180 : Ideal of R

% -- i181 : dim tmp2s

% -- o181 = -1

% -- i182 : tmp2s

% -- o182 = ideal 1

% -- o182 : Ideal of R

\item[---] The ramification of the double cover $\fX$ is a reduced divisor
  with singular locus of dimension at most $5$. (Note that the jacobian
  determinant of the   quadrics $(q_i)_{0\leq i\leq 9}$ is always zero
  because there is a factor $2$ in Euler's formula. So we computed the jacobian of
  $(w_9q_i)_{0\leq i\leq 9}$ and found $w_9^{10} J_{10}$ where $J_{10}$ is a
  polynomial of degree $10$ whose restriction to some $\PP_3$ was smooth.)
\item[---] $\mG$ is locally free of rank $4$, globally generated  with $h^0(\mG)=8$, $h^0(\mGdual)=0$
  so Corollary\ref{refG4mukai} is still valid.
% -- i51 : tutu=chooseGoodMinors(25,4,M810);

% -- o51 : Ideal of R

% -- i52 : tutuIX=IX+tutu;

% -- o52 : Ideal of R

% -- i53 : time tutuIXs=saturate tutuIX
% --      -- used 191.082 seconds

% -- o53 = ideal 1

\item[---] The Peskine variety $\Ypesk$ has the expected dimension $6$.
\end{itemize}

% tex hilbertPolynomial(IX,Projective=>false)
%$\frac{7}{2}i^{3}+3\,i^{2}+\frac{3}{2}i+2$
\end{exam}
\noindent\hspace{-0.1\linewidth}\resizebox{1.2\linewidth}{!}{\arraycolsep=5pt
  ${}^t
   \beta=\left(\begin{array}{m{5cm}|m{5cm}|m{5cm}|m{5cm}}
  \input{F8transposebeta}
  \end{array}\right)$}

\bibliography{genre16}

\end{document}

%% file: F8transposebeta.tex
$w_{2}$& 
$w_{3}$& 
$0$& 
$w_{6}$ \\ \hline

$w_{4}$& 
$w_{5}$& 
$w_{6}$& 
$w_{9}$ \\ \hline

$\left(a^{2}+a+1\right)w_{0}+\left(a^{2}+1\right)w_{2}+a^{2}w_{3}+a^{2}w_{4}+w_{5}$& 
$\left(a^{2}+a+1\right)w_{0}+a\,w_{2}+w_{3}+a^{2}w_{4}+\left(a^{2}+1\right)w_{5}+\left(a^{2}+1\right)w_{6}+a\,w_{7}+a\,w_{8}+\left(a^{2}+1\right)w_{9}$& 
$a^{2}w_{0}+w_{1}+\left(a+1\right)w_{2}+\left(a+1\right)w_{3}+w_{4}+\left(a^{2}+1\right)w_{5}+\left(a^{2}+a+1\right)w_{6}+a^{2}w_{7}+\left(a+1\right)w_{9}$& 
$\left(a+1\right)w_{0}+w_{1}+\left(a+1\right)w_{2}+a\,w_{3}+a^{2}w_{4}+\left(a+1\right)w_{5}+\left(a+1\right)w_{6}+\left(a+1\right)w_{7}+a\,w_{8}+\left(a+1\right)w_{9}$ \\ \hline

$\left(a^{2}+a\right)w_{0}+w_{1}+a\,w_{2}+a^{2}w_{3}+\left(a^{2}+a\right)w_{4}+\left(a^{2}+a\right)w_{5}+w_{6}$& 
$\left(a^{2}+a\right)w_{0}+\left(a+1\right)w_{8}$& 
$\left(a^{2}+a\right)w_{0}+\left(a^{2}+a+1\right)w_{1}+\left(a^{2}+1\right)w_{2}+w_{3}+a^{2}w_{4}+w_{5}+w_{6}+\left(a^{2}+a\right)w_{7}+\left(a^{2}+1\right)w_{9}$& 
$\left(a+1\right)w_{0}+\left(a^{2}+a\right)w_{8}$ \\ \hline

$w_{7}$& 
$w_{8}$& 
$w_{9}$& 
$0$ \\ \hline

$\left(a^{2}+a\right)w_{0}+\left(a^{2}+a\right)w_{1}+\left(a^{2}+1\right)w_{2}+\left(a^{2}+a\right)w_{3}+\left(a+1\right)w_{4}+\left(a+1\right)w_{5}+\left(a^{2}+a+1\right)w_{6}+a^{2}w_{7}+w_{8}$& 
$\left(a+1\right)w_{0}+\left(a^{2}+1\right)w_{1}+\left(a^{2}+a+1\right)w_{2}+\left(a^{2}+1\right)w_{3}+a\,w_{4}+\left(a^{2}+a\right)w_{5}+\left(a^{2}+1\right)w_{6}+\left(a+1\right)w_{7}+\left(a+1\right)w_{9}$& 
$a^{2}w_{3}$& 
$a^{2}w_{1}$ \\ \hline

$a^{2}w_{0}+\left(a^{2}+a+1\right)w_{2}+\left(a^{2}+a\right)w_{3}+a\,w_{4}+a^{2}w_{5}+\left(a^{2}+1\right)w_{6}+\left(a^{2}+a+1\right)w_{8}+w_{9}$& 
$a^{2}w_{0}$& 
$a^{2}w_{1}+a^{2}w_{5}$& 
$a^{2}w_{8}$ \\ \hline

$0$& 
$\left(a^{2}+1\right)w_{0}+\left(a^{2}+a\right)w_{1}+w_{2}+\left(a^{2}+a+1\right)w_{4}+\left(a^{2}+1\right)w_{5}+\left(a^{2}+1\right)w_{6}+\left(a+1\right)w_{7}+w_{8}+w_{9}$& 
$w_{0}+a^{2}w_{1}+\left(a^{2}+a+1\right)w_{2}+\left(a+1\right)w_{4}+\left(a^{2}+a\right)w_{5}+w_{6}+a^{2}w_{7}+w_{8}+a\,w_{9}$& 
$\left(a+1\right)w_{0}+a^{2}w_{1}+\left(a+1\right)w_{6}+\left(a^{2}+a\right)w_{7}+\left(a+1\right)w_{8}+a\,w_{9}$